\documentclass[11pt]{amsart}

\usepackage{authblk}

\usepackage[pdftex]{graphicx}
\usepackage[ddmmyyyy]{datetime}

\usepackage{amssymb,amsfonts,amsmath}
\usepackage{amsthm,xypic,enumerate}
\usepackage{color,tikz,pgf}
\usepackage{fancyhdr}
\input xy
\xyoption{all}
\usepackage{hyperref}

\newcommand{\N}{\mathbb{N}}
\newcommand{\R}{\mathbb{R}}
\newcommand{\mmS}{\mathcal{S}}
\newcommand{\mms}{\mathcal{S}}
\newcommand{\mmC}{\mathcal{C}}
\newcommand{\mmR}{\mathcal{R}}
\newcommand{\omR}{\overline{\R}}
\newcommand{\wG}{\widehat{G}}
\newcommand{\mmsa}{\mmS_{\alpha}}

\DeclareMathOperator*{\Con}{Con}

\newenvironment{enumerate*}[1][{}]{\begin{itemize}\renewcommand{\labelitemi}{#1}}{\end{itemize}}

\numberwithin{equation}{section}
\newtheorem{theorem}[equation]{Theorem}
\newtheorem{proposition}[equation]{Proposition}
\newtheorem{corollary}[equation]{Corollary}
\newtheorem{lemma}[equation]{Lemma}

\theoremstyle{definition}
\newtheorem{definition}[equation]{Definition}

\title{Variable elimination in chemical reaction networks with mass action kinetics}

\author{Elisenda Feliu, Carsten Wiuf}

\thanks{{\it Authors affiliation}: Bioinformatics Research Centre, Aarhus University, C. F. M\o llers All\'{e} 8, DK-8000 Aarhus, Denmark}
\thanks{{\it Corresponding author}: Elisenda Feliu, efeliu@birc.au.dk}

\date{\today}

\begin{document}

\maketitle

\begin{abstract}
We consider chemical reaction networks taken with mass action kinetics.  The steady states of such a system are solutions to a system of polynomial equations. Even for small systems the  task of finding the solutions is daunting. We  develop an algebraic  framework and procedure for linear elimination of variables. The procedure reduces the variables in the system to a set of ``core'' variables by eliminating variables corresponding to a set of non-interacting species.
The steady states are parameterized algebraically by the core variables, and a graphical condition is given for when a steady state with positive core variables necessarily have all variables positive.
Further, we characterize graphically the sets of eliminated variables that are constrained by a conservation law and show that this conservation law takes a specific form. 


\end{abstract}

\section{Introduction}
The goal of this work is to discuss linear elimination of variables at steady state in Chemical Reaction Networks (CRNs) taken with  mass action kinetics. We use  the formalism of Chemical Reaction Network Theory (CRNT) that puts CRNs into a mathematical, in particular algebraic, framework. CRNT was developed around 40 years ago, mainly by Horn, Jackson and Feinberg \cite{feinberg-first,feinbergnotes,feinberg-invariant,hornjackson}. Its usefulness for analysis of  CRNs is continuously being supported \cite{craciun-feinbergI,craciun-feinbergII,shinar-science}.

We introduce the elimination procedure  by going through a specific  example.  Consider five chemical species $A,B,C,D,E$ that interact according to the reactions:
$$\xymatrix@C=15pt{r_1:A+2B \ar[r] & D } \quad \xymatrix@C=15pt{r_2:D \ar[r] & A+C }\quad \xymatrix@C=15pt{r_3:C+D \ar[r] & E }\quad \xymatrix@C=15pt{r_4: E \ar[r] & A+B. } $$
The molar  concentration of a species $S$ at time $t$ is denoted by $c_S=c_S(t)$. By employing the common assumption  that  reaction rates are of \emph{mass action} type, the concentrations change with time according to a system of ordinary  differential equations (ODEs):
\begin{align*}
\dot{c_A} &= -k_{1}c_Ac_B^2 + k_{2}c_D + k_{4}c_E  & \dot{c_B} &= -2k_{1}c_Ac_B^2 +  k_{4}c_E  &
\dot{c_C} &= k_{2}c_D -k_{3}c_Cc_D \\ \dot{c_D} &= k_{1}c_Ac_B^2 - k_{2}c_D - k_{3}c_Cc_D  &
\dot{c_E} &=  k_{3}c_Cc_D -k_{4}c_E, 
\end{align*}
where  $k_{i}$ denotes the \emph{positive} rate constant  of reaction $r_i$.
Observe that $\dot{c_A}+\dot{c_D}+\dot{c_E}=0$, which implies that $c_A+c_D+c_E$ is constant over time, and fixed by the sum of  initial concentrations $c_0=c_A(0)+c_D(0)+c_E(0)$. The equation $c_0=c_A+c_D+c_E$ is called a \emph{conservation law}.

We are interested in the steady state solutions of this system, in particular,  the \emph{positive}  steady state solutions, that is,  the solutions for which all concentrations are positive. The steady state solutions are found by  setting the ODEs to zero. Consider  the equations $\dot{c_A}=0$, $\dot{c_D}=0$, and $\dot{c_E}=0$:
\begin{align}
0 &=  -k_{1}c_Ac_B^2 + k_{2}c_D + k_{4}c_E, & 0 &= k_{1}c_Ac_B^2 - k_{2}c_D - k_{3}c_Cc_D,  \label{ex:cut} \\
0 &=  k_{3}c_Cc_D -k_{4}c_E. \nonumber 
\end{align}
The equations form a system of polynomial equations in $c_A, c_B, c_C, c_D, c_E$ with real coefficients. None of the equations contain a monomial with more than one of the variables $c_A,c_D,c_E$. Further, the degree of $c_A,c_D,c_E$ is one in all equations. In other words, if we let $\Con=\{k_1,k_2,k_3,k_4\},$ then \eqref{ex:cut} is a linear system of equations in the variables $c_A,c_D,c_E$ with coefficients in the field $\R(\Con\cup\{c_B,c_C\})$:
$$\left(\begin{array}{ccc}
-k_1c_B^2 & k_2 & k_4 \\ 
 k_1 c_B^2 & -k_2- k_3c_C & 0  \\
0 &  k_3c_C & -k_4 \end{array}
\right)\left(\begin{array}{c} c_A \\ c_D \\ c_E \end{array}
\right)=0.$$
The column sums of the $3\times 3$ matrix are zero, because of the conserved amount $c_0$. In fact, the matrix has  rank $2$ in   $\R(\Con\cup\{c_B,c_C\})$ and the solutions of the system form a   line parameterized for example by $c_D$:
$$c_A=\frac{ k_{1}c_B^2 }{k_2+ k_3 } c_D,\qquad c_E=\frac{k_4}{ k_3c_C}c_D. $$ 
These solutions are well-defined in the field  $\R(\Con\cup\{c_B,c_C\})$. If positive values of  $c_B,c_C,c_D$ are given, then the steady state values of $c_A,c_E$ are positive and completely determined. Further, since   $c_A+c_D+c_E=c_0$, we find that
$$ c_D = c_0\left(1+\frac{ k_{1}c_B^2 }{k_2+ k_3 }  +\frac{k_4}{ k_3c_C} \right)^{-1}$$ 
and conclude that the positive steady states  are fully determined by the positive steady state solutions  of $c_B,c_C$. These solutions are found from the equations $\dot{c_B}=0$ and $\dot{c_C}=0$ in $c_B,c_C$, by substituting the values of $c_A,c_D,c_E$. Note that the equations can always be rewritten in polynomial form.

Let us now start from the equations $\dot{c_C}=0$ and $\dot{c_E}=0$:
$k_{2}c_D -k_{3}c_Cc_D=0$ and $k_{3}c_Cc_D -k_{4}c_E=0$.
This system is linear in the variables $c_C,c_E$ with coefficients in the field $\R(\Con\cup\{c_D\})$. Further, the  system has maximal rank in $\R(\Con\cup\{c_D\})$   and thus has a unique solution
$c_C=k_2/k_3$, $c_E=k_2c_D/k_4$ in $\R(\Con\cup\{c_D\})$.
As above, the variables $c_C,c_E$ can be eliminated and recovered from any positive steady state solution  $c_A,c_B,c_D$ of the remaining equations. 

The approaches that are used to eliminate the variables   in $\{c_A,c_D,c_E\}$ and $\{c_C,c_E\}$ differ:  In the former case the system is homogeneous, does not have maximal rank, and the conservation law is required for full elimination. In the latter the system has maximal rank but is not homogeneous. At this point we might ask: What are the similarities between the two sets of variables that  enable their elimination  from the steady state equations? What are the differences  that lead to different approaches? 
The species in both sets do not interact with each other, that is, they do not appear on the same side of a reaction, and further all concentrations have degree one in all the equations in which they appear. Such sets are called \emph{non-interacting}. However, in the first case the  sum of concentrations is conserved, and the set is what we call a \emph{cut}, while in the second case it is not.  Importantly, the eliminated variables are  non-negative whenever the non-eliminated variables (the \emph{core} variables) are positive. Furthermore, the cases in which zero concentrations of the eliminated variables can occur can be completely characterized. The concentration $c_B$ cannot be linearly eliminated because  it has degree $2$ in some equations.

This reaction system is  small  compared to real biochemical systems and can be manipulated manually. For an arbitrary  CRN, most non-interacting sets
can be eliminated using one of the approaches outlined above, depending on the presence or absence of a conserved amount.  After reduction, the (positive) steady states are  the solutions to polynomial equations depending on the core  variables only. Thus the steady states form  an algebraic variety in the core variables. 

In this manuscript we discuss a general procedure  for  linear elimination of variables, embracing the two approaches described above. The design of the procedure relies on (a specific version of) the \emph{species graph}. 
Subsets of species that can be eliminated are non-interacting. These sets correspond to a specific type of subgraphs and in any such set, the corresponding subgraph encodes the presence or absence of a conservation law relating the concentrations in the set. We study the interplay between non-interacting sets, subgraphs, and conservation laws and relate subgraph connectedness  to minimality of conservation laws and the existence of conservation laws to the so-called \emph{full} subgraphs. Thus, the results obtained here are of interest in their own. The Matrix-Tree theorem \cite{Tutte-matrixtree}  is key to study positivity of  solutions  \cite{TG-rational}. 

The elimination procedure has interesting potential applications. First of all, essential  information about the system at steady state is contained in the equations for the  core variables. Thus, experimental knowledge about the concentrations of the core variables is sufficient to explore the system at steady state. Further,  the species graph can be used in experimental planning  by choosing  (if possible)  a subgraph that optimizes the  information  in the experiment.

Secondly,  two different reaction systems involving the same chemical species can be discriminated based on the core  variables alone and thus used for model selection. This is possible, irrespectively whether the rate constants are known or not \cite{G-distributivity,MG-multisite}, if  the two algebraic varieties described by the core variables take different forms. 
A series of measurements with different initial concentrations can determine which variety the measurements belong to.

Finally, another potential application concerns the emergence of multiple positive steady states in a specific system. Many mathematical tools for detecting whether a system has at most one positive steady state exist  \cite{angelisontag2,craciun-feinbergI,feinberg-def0}. However, when these fail, it is not straightforward to conclude that the system  admits more than one positive steady state and  rate constants need to be found for which  this is true. This is typically done by performing a random parameter search. Elimination of variables might reduce the computational burden substantially and decrease the likelihood of numerical errors.

This work builds on our previous work on variable elimination  in so-called \emph{Post-Trans\-la\-tional Modification} (PTM) systems \cite{fwptm}. PTM systems form a  special type of biochemical reaction networks that are particularly  abundant in cell signaling and have been the focus of much theoretical research \cite{Heinrich-kinase,Huang-Ferrell,Markevich-mapk}. A subclass of  PTM systems was studied by Thomson and Gunawardena \cite{TG-rational}.  The present work extends the elimination  procedure for PTM systems to arbitrary CRNs. By doing so, some particularities of PTM systems are uncovered to be irrelevant.

The outline of the paper is the following. We introduce the notation, some preliminaries, and  CRNs together with their associated mass action ODEs. We proceed to  discuss  conservation laws arising from so-called semiflows, with special attention to  minimal semiflows. Next, the species graph and its relevant subgraphs (full and non-interacting) are defined, and we proceed to discuss relations between the subgraphs and semiflows.   We then present the variable elimination procedure  and the reduction of the steady state equations to a polynomial system  in the core variables. Using the graphical representation, we show that positive solutions of the core variables in the reduced system correspond to  non-negative steady states of the CRN, in which only the eliminated variables can possibly be zero.   

\section{Notation}
Let $\R_{+}$ denote the set of positive real numbers (without zero) and $\omR_{+}$  the set of non-negative real numbers (with zero). Given a finite set $\mathcal{E}$, let $\R^{\mathcal{E}}$ be the real vector space of formal sums $v=\sum_{E\in \mathcal{E}} \lambda_E E$, with $\lambda_E\in \R$. If $\lambda_E\in \R_{+}$ (resp.  $\omR_{+}$) for all $E\in \mathcal{E}$,  then we write $v\in \R_{+}^{\mathcal{E}}$ (resp. $\omR_{+}^{\mathcal{E}}$).

\medskip
{\bf S-positivity.} 
Let $\R[\mathcal{E}]$ denote the ring of real polynomials in $\mathcal{E}$.  A monomial is a  polynomial of the form
$\lambda  \prod_{E\in \mathcal{E}} E^{n_E}$ for some $\lambda\in \R\setminus \{0\}$ and  $n_E\in \N_0$ (the natural numbers including zero). A non-zero polynomial  in $\R[\mathcal{E}]$ with  non-negative coefficients is called \emph{S-positive}. Any assignment  $a\colon\mathcal{E}\rightarrow \R_+$ induces an evaluation map $e_a\colon\R[\mathcal{E}] \rightarrow \R$.  If $p\in \R[\mathcal{E}]$ is S-positive, then $e_a(p)> 0$.
 
A rational function $f$ in $\mathcal{E}$ is S-positive if it is a quotient of two S-positive polynomials in $\mathcal{E}$. Then $e_a(f)$ is well-defined and positive for any assignment  $a\colon \mathcal{E}\rightarrow \R_+$.  In general, a rational function $f=p/q$ in  $z_1,\dots,z_s$ and  coefficients in  $\R(\mathcal{E})$ is S-positive if  the coefficients of $p$ and $q$ are S-positive rational functions in $\mathcal{E}$. Then $e_a(f)$ is an S-positive rational function in $\R(z_1,\dots,z_s)$, for any assignment $a:\mathcal{E}\rightarrow \R_+$. Assume that  $E=g(\widehat{\mathcal{E}})$  for some rational function  $g$ in $\widehat{\mathcal{E}}=\mathcal{E}\setminus \{E\}$ and  $E\in \mathcal{E}$. Then, if $f$ is a rational function in $\mathcal{E}$, substituting $g$ into $f$ gives $f$ as a rational function in $\widehat{\mathcal{E}}$.

\medskip
{\bf Graphs and the Matrix-Tree theorem.} Let $G$ be a directed graph with node set $\mathcal{N}$. A \emph{spanning tree} $\tau$ of $G$ is a directed subgraph with node set $\mathcal{N}$ and such that the corresponding undirected graph is connected and acyclic. Self-loops are by definition excluded from a spanning tree.
There is a (unique) undirected path between any two nodes in a spanning tree \cite{Diestel}. We say that the spanning tree $\tau$ is \emph{rooted} at a node $v$ if  the unique path between any node $w$ and $v$ is directed from $w$ to $v$. As a consequence, $v$ is the only node in $\tau$ with no edges of the form $v\rightarrow w$ (called out-edges). Further, there is no node in $\tau$ with two out-edges.  The graph $G$ is \emph{strongly connected} if there is a directed path from $v$ to $w$ for any pair of nodes $v,w$. Any directed path from $v$ to $w$ in a strongly connected graph can be extended to a spanning tree rooted at $w$. Some general references for graph theory are \cite{Diestel} and \cite{Gross-Yellen}.

If $G$  is  labeled, then $\tau$ inherits a labeling from $G$ and we define
$$\pi(\tau)=\prod_{x\xrightarrow{a}y \in \tau} a.$$
Assume that $G$ has no self-loops. We order the node set $\{v_1,\dots,v_n\}$ of $G$  and let $a_{i,j}$ be the label of the edge $v_i\rightarrow v_j$. Further, we set $a_{i,j}=0$ for $i\neq j$ if there is no edge from $v_i$ to $v_j$ and $a_{i,i}=0$. Let $\mathcal{L}(G)=\{\alpha_{i,j}\}$ be the \emph{Laplacian} of $G$, that is, the matrix with
$\alpha_{i,j} =a_{j,i}$ if $i\neq j$ and $\alpha_{i,i} =-\sum_{k=1}^n a_{i,k}$,
such that the column sums are zero. Any matrix whose column sums are zero can be realized as the Laplacian of a directed labeled graph with no self-loops.

For each node $v_{j}$, let $\Theta(v_j)$ be the set of spanning trees of $G$ rooted at $v_{j}$. Let $\mathcal{L}(G)_{(ij)}$ denote the  determinant of the principal minor of $\mathcal{L}(G)$  obtained by removing the $i$-th row and  the $j$-th column of $\mathcal{L}(G)$. Then, by the Matrix-Tree theorem \cite{Tutte-matrixtree}:
$$ \mathcal{L}(G)_{(ij)} = (-1)^{n-1+i+j}  \sum_{\tau \in \Theta(v_j)}  \pi(\tau).$$ 
Note that for notational simplicity we have defined the Laplacian as the transpose of how it is   usually defined and the Matrix-Tree theorem has been adapted consequently. 

\section{Chemical reaction networks}\label{CRN}

We introduce the definition of a CRN and some concepts related to CRNs. See for instance \cite{feinbergnotes,Feinbergss} for extended discussions.

\begin{definition}\label{crn}
A \emph{chemical reaction network} (CRN) consists of three finite sets:
\begin{enumerate}[(1)]
\item A set $\mmS$ of \emph{species}.
\item A set $\mmC\subset \omR_{+}^{\mmS}$ of \emph{complexes}.
\item A set  $\mmR\subset \mmC\times \mmC$ of \emph{reactions}, such that $(y,y)\notin \mmR$ for all $y\in \mmC$, and if $y\in \mmC$, then there exists $y'\in \mmC$ such that either $(y,y')\in \mmR$ or $(y',y)\in \mmR$. 
\end{enumerate}
\end{definition}

Inflow and outflow of species are accommodated in this setting by incorporating the complex $0\in \omR_+^{\mmS}$ and reactions $0\rightarrow  A$, $A\rightarrow0$, respectively \cite{feinberg-horn-open}.

Following the usual convention, an element $r=(y,y')\in \mmR$ is  denoted  by $r\colon y\rightarrow y'$.
For a reaction $r\colon y\rightarrow y'$, the initial and terminal complexes are denoted by  $y(r):=y$ and $y'(r):=y'$, respectively. By definition,   any complex is either the initial or terminal complex of some reaction.  

Let $s$ be the cardinality of $\mmS$.  We fix an order in $\mmS$ so that $\mmS=\{S_1,\dots,S_s\}$ and  identify $\R^{\mmS}$  with $\R^s$. The species $S_i$ is identified with the $i$-th canonical vector of $\R^s$ with $1$ in the $i$-th position and zeroes elsewhere. An element in $\R^s$ is then given as $\sum_{i=1}^s \lambda_i S_i$.  In particular, a complex $y\in \mmC$ is given as $y=\sum_{i=1}^s y_i S_i$ or  $(y_1,\dots,y_s)$. If $r$ is a reaction,  $y_i(r), y'_i(r)$  denote the $i$-th entries of $y(r), y'(r)$ respectively. 

\begin{definition}\label{maindefs} We say:
\begin{enumerate}[(i)]
\item $y_i$ is the \emph{stoichiometric coefficient} of $S_i$ in $y$.
\item If $y_i\neq 0$ for some $i$ and $y\in \mmC$, then $S_i$ \emph{is part of} $y$, $y$ \emph{involves} $S_i$, and $r$ \emph{involves} $S_i$ for any reaction $r$ such that $y(r)=y$ or $y'(r)=y'$.
\item $S_i,S_j\in \mmS$  \emph{interact} if $y_i,y_{j} \neq 0$, $i\not=j$, for some complex $y$.
\item $y\in \mathcal{C}$ \emph{reacts to}  $y'\in\mathcal{C}$ if there is a reaction $y\rightarrow y'$.
\item $y\in \mathcal{C}$ \emph{ultimately reacts to}  $y'\in\mathcal{C} $ (denoted $y\Rightarrow y'$) if there exists a sequence of reactions  $y \rightarrow y^{1} \rightarrow \dots \rightarrow y^{r} \rightarrow y'$ with $y^{m}\in \mathcal{C}$. 
\item $S_i\in \mmS$ \emph{produces} $S_j\in \mmS$  if there exist two complexes $y,y'$ with $y_i\neq 0$, $y'_{j}\neq 0$  and a reaction $y\rightarrow y'$. 
\item  $S_i\in \mmS$ \emph{ultimately produces} $S_j\in \mmS$  if there exist $S_{i_1},\dots,S_{i_r}$ with $S_{i_1}=S_i$, $S_{i_r}=S_j$ and such that $S_{i_{k-1}}$ produces $S_{i_k}$ for $k=2,\dots,r$.
If each $S_{i_k}$ belongs to a  subset $\mmsa\subseteq \mms$ for $k=2,\dots,r-1$, then  $S_i$ \emph{ultimately produces} $S_j$ \emph{via} $\mmsa$.
\end{enumerate}

\end{definition}

If $y$ ultimately reacts to $y'$ then $y$ and $y'$ are \emph{linked}. Being linked generates an equivalence relation and the classes are called  \emph{linkage classes}. Two complexes $y,y'$ are \emph{strongly linked} if both $y\Rightarrow y'$ and $y'\Rightarrow y$. Being strongly linked also defines  an equivalence relation and the classes  are called \emph{strong linkage classes}.

We introduce an example (which we will refer to as the \emph{main example}) that we use to illustrate the definitions and constructions below.
Consider the  CRN with set of species $\mmS=\{S_1,\dots,S_9\}$ and set of complexes $\mmC=\{S_1 + S_2,S_4,S_1+S_3,S_5,S_3+S_4,S_6,S_2+S_5,S_1+S_7,S_7+S_8,S_9,S_2+S_3+S_8\}$,
reacting according to
\begin{center}
\xymatrix@R=10pt@C=10pt{
S_1+ S_2 \ar@<0.4ex>[r] & S_4 \ar@<0.4ex>[l] & S_1+ S_3 \ar@<0.4ex>[r] & S_5 \ar@<0.4ex>[l] & S_3+S_4\ar@<0.4ex>[r] & S_6 \ar[d] \ar@<0.4ex>[l]  \ar@<0.4ex>[r] & S_2+S_5 \ar@<0.4ex>[l]   \\ & S_7+ S_8 \ar@<0.4ex>[r] & S_9 \ar[r] \ar@<0.4ex>[l] & S_2+S_3+S_8 & &  S_1+S_7}
\end{center}
That is, the set of reactions $\mathcal{R}$ consists of
\begin{align*}
r_1\colon & S_1 + S_2 \rightarrow  S_4 & r_2 \colon & S_4 \rightarrow S_1 + S_2 & r_3 \colon & S_1+S_3\rightarrow S_5 
& r_4\colon & S_5 \rightarrow S_1+S_3 \\ 
r_5 \colon & S_3+S_4 \rightarrow S_6 & r_6\colon &  S_6 \rightarrow S_3+S_4   & r_7\colon & S_2 + S_5 \rightarrow S_6 & r_8\colon & S_6\rightarrow S_2+S_5 \\  
  r_9 \colon & S_6 \rightarrow S_1+S_7 &
r_{10}\colon & S_7+S_8 \rightarrow S_9 & r_{11}\colon  & S_9 \rightarrow S_7+S_8  & r_{12} \colon & S_9\rightarrow S_2+S_3+S_8
\end{align*}
This system represents  a two substrate enzyme catalysis with unordered substrate binding \cite{craciun-feinberg-pnas} in which $S_1$ is an enzyme, $S_2,S_3$ are substrates, $S_4,S_5,S_6$ are intermediate enzyme-substrate complexes, and $S_7$ is considered the product of the reaction system. The product dissociates via catalysis by an enzyme $S_8$ and the formation of an intermediate complex $S_9$.
The stoichiometric coefficients of all species that are part of a complex are one. The complex $S_2+S_3+S_8$ involves $S_2,S_3,S_8$ and its vector expression is $(0,1,1,0,0,0,0,1,0)\in \omR_{+}^9$. The species $S_3$ and $S_4$ interact. The complex $S_1+S_3$ reacts to the complex $S_5$, implying that species $S_1$ produces species $S_5$. Also, the complex $S_3+S_4$ ultimately reacts to $S_1+S_7$, and  $S_7+S_8$ ultimately reacts to $S_2+S_3+S_8$. It follows that $S_3$ ultimately produces $S_7$ and $S_8$.

\section{Mass-action kinetics} 
The molar concentration  of species $S_i$ at time $t$ is denoted by $c_i=c_i(t)$. To any complex $y$ we associate a monomial $c^y=\prod_{i=1}^s  c_i^{y_i}$. For example, if $y=(2,1,0,1)\in \omR^4_+$, then the associated monomial is $c^y=c_1^2c_2c_4$. 

We assume that each reaction $r:y\rightarrow y'$ has an associated positive  \emph{rate constant} $k_{y\rightarrow y'}\in \R_+$ (also denoted $k_r$). The set of reactions together with their associated rate constants give rise to a polynomial system of ODEs taken with \emph{mass action kinetics}:
\begin{align}\label{ode}
\dot{c_i} &=\sum_{y\rightarrow y'\in \mmR} k_{y\rightarrow y'} c^y (y_i' -   y_i),\qquad S_i\in \mmS.
\end{align}
These ODEs describe the dynamics of the concentrations $c_i$ in time.
The steady states of the system are the solutions to a system of polynomial  equations in $c_1,\dots,c_s$ obtained by setting the derivatives of the  concentrations to zero:
\begin{align}\label{steadystate}
0 =&\sum_{y\rightarrow y'\in \mmR} k_{y\rightarrow y'} c^y (y_i' -   y_i), \qquad \textrm{for all }i.
\end{align}
These polynomial equations can be written as:
\begin{equation}\label{steadystate2}
 0=\sum_{r\in \mmR} k_{r}  c^{y(r)} y'_i(r)  -\sum_{r\in \mmR} k_{r}c^{y(r)} y_i(r),\qquad \textrm{for all }i. 
 \end{equation}

It is convenient to treat  the rate constants as parameters with unspecified  values, that is as symbols (as we did in the example in the introduction). For that, let 
$$\Con=\{k_{y\rightarrow y'}| y\rightarrow y'\in \mmR \}$$
 be the set of the symbols. Then, the system \eqref{steadystate}  is a system of polynomial  equations in $c_1,\dots,c_s$ with coefficients in the field $\R(\Con)$. 
 
Only  \emph{non-negative solutions} of the steady state equations are biologically or chemically meaningful and we focus on these only. The concept of S-positivity introduced above will be key in what follows.  
Consider the main example and denote by $k_i$ the rate constant of reaction $r_i$. The mass action ODEs are:
\begin{align*}
\dot{c_1} &= -k_1 c_1c_2 + k_2c_4 - k_3c_1c_3+k_4c_5+k_9c_6 & \quad \dot{c_5}&= k_3c_1c_3 - k_4c_5-k_7c_2c_5+k_8c_6 \\
\dot{c_2} &=  -k_1 c_1c_2 + k_2c_4 -k_7c_2c_5+k_8c_6+k_{12}c_9 &  \dot{c_7} &= k_9c_6 - k_{10} c_7c_8 + k_{11}c_9\\
\dot{c_3}&=  - k_3c_1c_3+k_4c_5 - k_5c_3c_4 + k_6 c_6+k_{12}c_9 &  \dot{c_8} &= -k_{10} c_7c_8 + k_{11}c_9 + k_{12}c_9\\
\dot{c_4} &=  k_1 c_1c_2 - k_2c_4 - k_5c_3c_4 + k_6 c_6 &  \dot{c_9} &= k_{10} c_7c_8 - k_{11}c_9 - k_{12}c_9.\\
\dot{c_6} &= k_5c_3c_4 - k_6 c_6 + k_7c_2c_5 - k_8c_6-k_9c_6 
\end{align*}
Take for instance species $S_1$. The only reactions that involve $S_1$ are $r_1,r_2,r_3,r_4,r_9$. $r_1,r_3$ involve $S_1$ in the initial complex and thus the monomials contain $c_1$ and have negative coefficients. Similarly, $r_2,r_4,r_9$ involve $S_1$ only in the terminal complex and thus the monomials do not include $c_1$ and have positive coefficients.

\section{Conservation laws and P-semiflows}
The dynamics of a CRN system might preserve quantities that remain constant  over time. If this is the case, the dynamics  takes place in a proper invariant subspace of $\R^{s}$.  Let $x\cdot x'$ denote the Euclidian scalar product of two vectors $x,x'$.

\begin{definition}
The \emph{stoichiometric subspace}  of a CRN, $(\mmS,\mmC,\mmR)$, is the following subspace of $\R^s$: 
$$\Gamma = \langle y'-y|\, y\rightarrow y' \in \mmR\rangle.   $$
A  \emph{semiflow} is a non-zero vector $\omega=(\lambda_1,\dots,\lambda_s)\in \Gamma^{\perp}$. If $\lambda_i\geq 0$ for all $i$, then $\omega$ is  a \emph{P-semiflow}.
\end{definition}

By the definition of the mass action ODEs, the vector $\dot{c}$ points along the stoichiometric 
subspace $\Gamma$. The \emph{stoichiometric class} of a concentration vector $c\in \omR^s_{+}$ is $\mathcal{C}_{c}=\{c+\Gamma\}\cap\omR_{+}^s$.  In CRNT, two steady states    $c,c'$ are called \emph{stoichiometrically compatible} if $c-c'\in \Gamma$. This is equivalent to $\omega\cdot c = \omega \cdot c'$ for all $\omega\in \Gamma^{\perp}$. 

If $\omega=(\lambda_1,\dots,\lambda_s)\in \Gamma^{\perp}$, then  $\sum_{i=1}^s \lambda_i \dot{c_i}=0$. This implies that the linear combination of  \emph{concentrations} $\sum_{i=1}^s \lambda_i c_i$ is independent of time and thus determined by the initial concentrations of the system. 
In particular, any steady state solution of the system preserves the total initial amounts and lies in a particular coset of $\Gamma$. 

A linear combination  $\sum_{i=1}^s \lambda_ic_i$ that is  independent of time gives rise to an equation, called a \emph{conservation law}, with a fixed \emph{total amount} $\overline{\omega}\in \R_+$:
 \begin{equation} \label{totalamounts}
\overline{\omega} =\sum_{i=1}^s \lambda_i c_i. 
\end{equation}

A basis $\{\omega^1,\dots,\omega^d\}$ of $\Gamma^{\perp}$ gives a set of independent semiflows and thus a set of independent conservation laws: if $\omega^l=\sum_{i=1}^s\lambda_i^lS_i$ and  total amounts $\overline{\omega}^1,\dots,\overline{\omega}^d\in \R_{+}$ are given,  we require the steady state solutions to  satisfy:
$\overline{\omega}^l =\sum_{i=1}^s \lambda_i^l c_i$ for all $l$.

In the main example, the dimension of $\Gamma$ is $5$:
$$\Gamma=\langle S_1+S_2-S_4,S_1+S_3-S_5,S_3+S_4-S_6,S_1+S_7-S_6,S_7+S_8-S_9 \rangle.$$
Thus the space $\Gamma^{\perp}$ has dimension $4$ and a basis is:
\begin{align}\omega^1& =S_1+S_4+S_5+S_6 & \omega^2 &= S_8+S_9 \label{main:conslaws}\\ \omega^3 &= S_2+S_4+S_6+S_7+S_9 & \omega^4 &= S_2+S_3+S_4+S_5+2S_6+2S_7+2S_9.\nonumber
\end{align}
The conservation law corresponding to $\omega^1$ is
$\overline{\omega}^1= c_1+c_4+c_5+c_6$
for a given $\overline{\omega}^1\in \R_+$.

\emph{Remark. } Questions like ``How many steady states does a system possess?'' refer to the number of steady state solutions that fulfill the conservation laws with the same total amounts, or, equivalently, to the number of (stoichiometrically compatible) steady states in each stoichiometric class. If this restriction is not imposed and conservation laws exist, then the steady state solutions form an algebraic variety of dimension at least $1$.

\emph{Remark. } Not all CRNs have semiflows. Consider for instance the  CRN with $s=6$ and reactions
$S_1\rightarrow S_2$ , $S_2\rightarrow  S_3$, $S_1+S_2+S_3 \rightarrow S_6$, $S_4+S_5\rightarrow  S_6$, $S_4\rightarrow  S_5$, and $S_5\rightarrow  S_1$.
The stoichiometric subspace is $\R^6$ and thus $\Gamma^{\perp}=0$. If the last reaction is removed, then the stoichiometric subspace has dimension one and there is one P-semiflow: $\omega=2S_1+2S_2+2S_3+3S_4+3S_5+6S_6$. In general,  a basis for $\Gamma^{\perp}$ consisting of P-semiflows is neither guaranteed.  Consider the following  CRN with $s=3$ and one reaction,  $A+B+C\rightarrow A$. There is not a basis of $\Gamma^{\perp}=\langle A, B-C\rangle$ consisting of P-semiflows alone. This CRN  is not biochemically reasonable.

\begin{lemma}\label{cons}
The following statements are equivalent:
\begin{enumerate}[(i)]
\item The stoichiometric class $\mathcal{C}_c=\{c+\Gamma\}\cap\omR_{+}^s$, $c\in\omR_{+}^s$, is compact.
\item $\Gamma\cap\omR_{+}^s=\{0\}$.
\item $\Gamma^{\perp}$ has a basis $\{\omega^1,\ldots,\omega^d\}$ of P-semiflows with $\lambda^j_i>0$ if $\omega^j=(\lambda_1^j,\ldots,\lambda_s^j)$.
\item There is an element  $\omega=(\lambda_1,\ldots,\lambda_s)$ of $\Gamma^{\perp}$ with $\lambda_i>0$ for all $i$.
\end{enumerate}
\end{lemma}
\begin{proof} We will prove (i)$\Rightarrow$(ii)$\Rightarrow$(iii)$\Rightarrow$(iv)$\Rightarrow$(i). Assume that $\mathcal{C}_c$  is compact and consider $V:=\Gamma\cap\omR_{+}^s$. If $V\not=\{0\}$, then for any non-zero $v\in V$, the set  $c+\langle v\rangle$ is unbounded in $\omR_{+}^s$. Hence $\mathcal{C}_c$ cannot be compact and thus (ii) must be the case.  If (ii), then $\Gamma^{\perp}\cap\omR_{+}^s\not=\{0\}$ and further $\Gamma^{\perp}\cap\omR_{+}^s\not\subseteq{\rm bd}(\omR_{+}^s)$.  If the latter was not the case, then also $\Gamma\cap{\rm bd}(\omR_{+}^s)\not=\{0\}$, contradicting (ii). Hence, there exists an open set $\Omega\subseteq\Gamma^{\perp}\cap\R_{+}^s$ in $\Gamma^{\perp}$,   and we can choose a basis $\{\omega^1,\ldots,\omega^d\}$ of P-semiflows with $\omega^j\in\Omega$, that is, $\lambda_i^j>0$. Thus (iii) is fulfilled. (iii) gives (iv) directly. Assume (iv). For $x=(x_i)_i\in\mathcal{C}_c$,  $\omega\cdot x=\omega\cdot c$ is independent of $x$. Since $\lambda_i>0$ and $x_i\geq 0$,  $x_i\leq (\omega\cdot c)/\lambda_i$ for all $i$ and thus $\mathcal{C}_c$ is bounded. Since it is a closed set, $\mathcal{C}_c$ is compact and (i) is proven. 
\end{proof}

Lemma~\ref{cons} is well-known in dynamical systems theory and Petri Net theory. In the latter semiflows are known as P-invariants (place invariants) \cite{petri}.

\emph{Remark. } All conservation laws might not be obtained from semiflows \cite{feinberg-invariant}, that is, the semiflows in $\Gamma^{\perp}$ might not give the minimal affine space in which the dynamics of the system takes place. There can be additional conservation laws  depending on the rate constants and not merely on the stoichiometric coefficients. The next lemma is proven in \cite{feinberg-invariant} and stated here for future reference:

\begin{lemma}[\cite{feinberg-invariant}, $\S$6] \label{stoichkin}
If each linkage class contains exactly one terminal strong linkage class, then  all conservation laws correspond to semiflows. 
\end{lemma}

As shown in \cite{feinberg-invariant}, any weakly reversible network fulfills the condition of the lemma. 
Also, the main example fulfills the criterion. The CRN  with reactions 
$r_1\colon S_1\rightarrow S_2$, $r_2\colon S_1\rightarrow S_3$ and $r_3\colon S_2+S_3\rightarrow 2S_1$ does not fulfill it  \cite{craciun-feinberg}. Here, $\Gamma^{\perp}=\langle S_1+S_2+S_3\rangle$ providing the conservation law $c_1+c_2+c_3=\overline{\omega}$. However, when $k_1=k_2=k_3$, then $c_1+2c_2$ is also conserved.

\medskip
{\bf Minimal and terminal semiflows.}

\begin{definition}\label{minimal} The \emph{support} of a semiflow $\omega=(\lambda_1,\dots,\lambda_s)$ is the set $\mmS(\omega)=\{S_i|\ \lambda_i\neq 0\}$. We say that  $\omega$ is 
\begin{enumerate}[(i)]
\item  \emph{minimal} if for any semiflow $\widetilde{\omega}$ with $\mms(\widetilde{\omega})\subseteq\mms(\omega)$, there is $a\in\R$  such that $a\widetilde{\omega}=\omega$. 
\item \emph{terminal} if any semiflow $\widetilde{\omega}$ with $\mms(\widetilde{\omega})\subseteq \mms(\omega)$ satisfies $\mmS(\widetilde{\omega})=\mmS(\omega)$.
\end{enumerate}
 \end{definition}
 
That is, a semiflow $\omega$ is minimal if  any semiflow given by a linear combination of the species in its support  is a multiple of $\omega$ and terminal if there is no semiflow with smaller support.

\begin{lemma}\label{terminal} 
\begin{enumerate}[(i)]
\item A semiflow is minimal if and only  if it is terminal. 
\item If $\omega$ is a P-semiflow that is not minimal, then  there is a P-semiflow $\widetilde\omega$ such that $\mms(\widetilde{\omega})\subsetneq \mms(\omega)$.
\end{enumerate}
\end{lemma} 
\begin{proof} (i) If $\omega$ is a minimal semiflow then by definition any semiflow $\widetilde\omega$ with $\mms(\widetilde{\omega})\subseteq \mms(\omega)$ satisfies $\widetilde\omega=a\omega$ for some $a\in \R$. Thus, $\mms(\widetilde\omega)=\mms(\omega)$, which implies that $\omega$ is terminal.  
To prove the reverse, assume that $\omega$  is terminal but not minimal, that is, there exists $\widetilde{\omega}$ such that $\mms(\widetilde{\omega})=\mms(\omega)$ and $\widetilde{\omega}\neq a\omega$ for all $a\in \R$.  Let $\mathcal{I}=\{i|S_i\in \mms(\widetilde{\omega})\}$, $\omega=(\lambda_1,\dots,\lambda_s)$, and $\widetilde\omega=(\widetilde\lambda_1,\dots,\widetilde\lambda_s)$. Choose $u\in \mathcal{I}$ such that $|\widetilde{\lambda}_u/\lambda_u|\geq |\widetilde{\lambda}_i/\lambda_i|$ for all $i\in \mathcal{I}$ and define $\gamma=\widetilde{\lambda}_u$ and $\widetilde\gamma=\lambda_u$.  Then
$$\widehat{\omega}:=\gamma\omega - \widetilde\gamma\widetilde \omega = \sum_{i=1}^s (\widetilde{\lambda}_u\lambda_i-\lambda_u\widetilde{\lambda}_i)S_i=\sum_{i=1}^s\mu_i S_i$$ is a semiflow, since $\widehat{\omega}\neq 0$ (otherwise $\widetilde{\omega}= a\omega$ for some $a$). Since  $\mu_u=0$, $\mms(\widehat{\omega})\subsetneq\mms(\omega)$, which contradicts that $\omega$ is terminal. 

(ii)  If $\omega$ is a P-semiflow that is not minimal, then there exists a semiflow $\widetilde{\omega}$ such that $\mms(\widetilde{\omega})\subsetneq \mms(\omega)$. The construction above provides a new semiflow $\widehat{\omega}$. Since  $\lambda_i> 0$ for all $i\in \mathcal{I}$, we have $|\widetilde{\lambda}_u|\lambda_i-\lambda_u|\widetilde{\lambda}_i|\geq 0$ and either $\widetilde{\lambda}_u>0$ and $\mu_i\geq 0$ for all $i$, or  $\widetilde{\lambda}_u<0$ and $\mu_i\leq 0$ for all $i$. Hence,  either $\widehat{\omega}$ or $-\widehat{\omega}$ is a P-semiflow fulfilling (ii).
\end{proof}

Therefore, there cannot exist two  linearly independent minimal P-semiflows with the same support. For example, if $S_1+S_2+S_3$ is conserved and minimal, then $\lambda_1S_1+\lambda_2S_2+\lambda_3S_3$ with $\lambda_1\neq \lambda_2$ cannot be a semiflow. 
We will see below that the P-semiflows $\omega^1,\omega^2,\omega^3$ in \eqref{main:conslaws} of the main example are  minimal. However, $\omega^4$ is not minimal since $\mmS(\omega^3)\subsetneq \mmS(\omega^4)$.

 The species graph does not characterize the CRN uniquely, since information coming from the stoichiometric coefficients is ignored. For instance, the following two systems have the same species graph:
  \begin{equation}
R_1=\{A+B\rightarrow  2C, C\rightarrow A\},\qquad R_2= \{A+B\rightarrow  C, C\rightarrow A\}. \label{system2}
\end{equation}

\section{Species graph} Given a CRN $(\mmS,\mmC,\mmR)$, we define the \emph{species graph} $G_{\mmS}$ as the labeled directed  graph with node set $\mmS$ and a directed edge from  $S_i$ to $S_j$ with label $r\colon y\rightarrow y'$ whenever $y_i\neq0$ and $y'_{j}\neq 0$.  That is, there is a directed edge from $S_i$ to $S_j$ if and only if $S_i$ produces $S_j$. There can be multiple edges with different labels between a pair of nodes. In addition,   if $S_i$ is involved in the initial and the terminal complexes of a reaction, then there is a self-edge $S_i\rightarrow S_i$.
The species graph of the main example is depicted in Figure \ref{main:graph}.

\emph{Remark. } A reaction $r\colon y\rightarrow y'$ is called \emph{reversible} if the reaction $y'\rightarrow y$ also exists.  
In the main example, all reactions but $r_9,r_{12}$ are reversible.
In contrast to other papers \cite{angelisontag2,TG-rational},  we consider  reversible reactions as two (independent) irreversible reactions. Thus, reversible reactions provide two edges with opposite directions and different labels   in the species graph. This is required when we consider spanning trees in Section \ref{sec:elim}. Changing a reaction from being reversible to irreversible does not change the stoichiometric subspace and  a system with all reactions considered  irreversible  has the same (P-)semiflows as a system with some (all) reactions considered reversible. However,  the steady states might depend on whether reactions are reversible or not.

\begin{figure}[!t]
\centering
\includegraphics[]{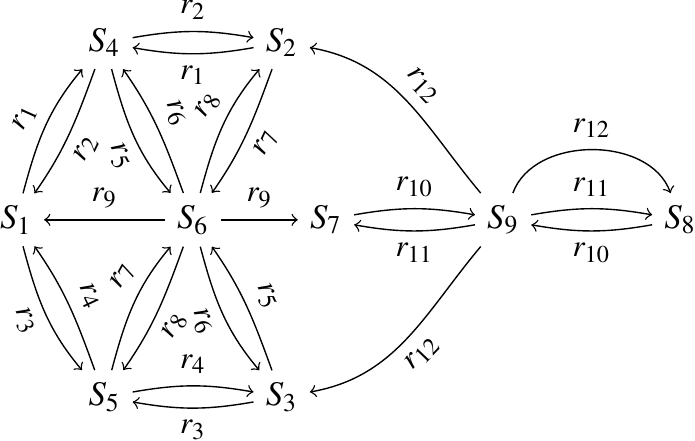}
\caption{Species graph of the main example.} \label{main:graph}
\end{figure}

\begin{definition} 
A graph $G$ with node set $\mmsa$ is a \emph{subgraph} of $G_{\mmS}$ if $\mmsa\subseteq \mmS$
and the labeled directed edges  of $G$ are inherited from $G_{\mmS}$. We denote $G=G_{\mmsa}$.
Further,  
\begin{enumerate}[(i)]
\item  $\mmsa$  is  \emph{full} if any reaction involving $S_i\in \mmsa$ appears at least once as a label of an edge in $G_{\mmsa}$. If this is the case, then  $G_{\mmsa}$ is said to be \emph{full}.
\item  $\mmsa$ is  \emph{non-interacting} if it contains no pair of interacting species  and all stoichiometric coefficients are  either $0$ or $1$, that is, $y_i=0,1$ for all $S_i\in \mmsa$ and $y\in \mmC$. If this is the case, then  $G_{\mmsa}$ is said to be \emph{non-interacting}.
\item If  $\mmsa$ is full and non-interacting, then $\mmS_{\alpha}$ is  a \emph{cut} of $\mmS$.
\end{enumerate}
\end{definition}

The definition of subgraphs of $G_{\mms}$ extends to subgraphs of $G_{\mmsa}$.
We depict in Figure \ref{main:subgraphs} four different subgraphs of the species graph of the main example, corresponding to four different subsets $\mmS_{\alpha}\subseteq \mmS$.

The proof of  the following lemma is left to the reader.
\begin{lemma}\label{snit}  Let $\mms_1,\mms_2\subseteq \mms$. Then
\begin{enumerate}[(i)]
\item  If $G_{\mms_1}$ and $G_{\mms_2}$ are full, then so is $G_{\mms_1\cup \mms_2}$.
\item If $G_{\mms_1\cup \mms_2}$ is  non-interacting, then so are $G_{\mms_1}$ and $G_{\mms_2}$.
\end{enumerate}
If $G_{\mms_1}\cap G_{\mms_2}=\emptyset$ and $G_{\mms_1\cup \mms_2}=G_{\mms_1}\cup G_{\mms_2}$, then the reverse statements are also true.
\end{lemma}

It follows that if $\mmsa$ is a cut, then the node set of  any connected component of $G_{\mmsa}$ is  also a cut (as illustrated in Figure \ref{main:subgraphs}(a)). The next lemma connects some properties of full and non-interacting graphs that will be used in the sequel.  A non-empty subgraph $G_{\mmsa'}$ of $G_{\mmsa}$ is \emph{proper} if $G_{\mmsa'}\neq G_{\mmsa}$.

\begin{lemma}\label{cut} Let $\mmsa\subseteq \mmS$ be a subset.
\begin{enumerate}[(i)]
\item If $G_{\mmsa}$ is non-interacting, then any reaction label $r$ appears at most once in $G_{\mms_{\alpha}}$.
\item If $G_{\mmsa}$ is non-interacting and $S_i\in \mmsa$ is involved in a reaction $r$ that is a label of an edge of $G_{\mmsa}$, then the edge is to/from $S_i$.
\item If $G_{\mmsa}$ is full and $S_i\in \mmsa$ is involved in a reaction $r$, then there is an edge to/from $S_i$ labeled  $r$.
\item  If $G_{\mmsa}$ is full, has no repeated reaction labels and the stoichiometric coefficients of its nodes in all complexes are either $0$ or $1$, then $\mmsa$ is a cut.
\item If $G_{\mmsa}$ has no repeated reaction labels and $G_{\mmsa}$ is connected, then $G_{\mmsa}$ has no proper full subgraphs.
\end{enumerate}
\end{lemma}
\begin{proof} (i)  Assume that a reaction $r$ appears in two different edges $S_i\xrightarrow{r} S_j$ and $S_u\xrightarrow{r}S_v$ of $G_{\mmsa}$.
If  $S_i\neq S_u$ or $S_j\neq S_v$, then either $S_i$ and $S_u$  or $S_j$ and $S_v$ interact and thus $G_{\mmsa}$ is not non-interacting. 

(ii)-(iii) In both cases there is an edge in $G_{\mmsa}$ labeled $r$: In (ii) by assumption and in (iii) because $G_{\mmsa}$ is full.  Assume that the edge with label $r$ is not from/to $S_i$ but between $S_j,S_u\in \mmsa$, $j,u\neq i$. Then,  $S_i$ and $S_j$ (or $S_u$) interact reaching a contradiction in case (ii) and implying that there is an edge with label $r$ between $S_i$ and $S_u$ (or $S_j$) in case (iii).

(iv) Assume that there are two (different) interacting nodes $S_i,S_j\in \mmsa$. Then there exists a reaction $r\colon y\rightarrow y'$ such that   $y_i=y_j=1$ or $y'_i=y'_j=1$. Let us assume that $y_i=y_j=1$ and the other case follows by symmetry. Since $G_{\mmsa}$ is full, $r$ is the label of an edge in $G_{\mmsa}$. Let $S_u$ (potentially equal to $S_i$ or $S_j$) be the end node of the edge. Since $S_i,S_j,S_u\in \mmsa$, the edges $S_i\xrightarrow{r} S_u$ and $S_j\xrightarrow{r} S_u$ are in  $G_{\mmsa}$, contradicting that there are no repeated labels. Hence, $\mmsa$ is non-interacting and hence a cut.

 (v) Assume that there is a proper subgraph $G$.  Since $G_{\mmsa}$ is connected one can find a node $S_i$ in $G_{\mmsa}$ that is not in $G$, and a node $S_j$ in $G$ for which there is an edge $S_i\xrightarrow{r} S_j$ or $S_j\xrightarrow{r} S_i$ in $G_{\mmsa}$.  By assumption a label appears at most  once in $G_{\mmsa}$. Thus $G$ is not full since $r$ is not a label of an edge in $G$, but $r$ involves $S_j\in G$. 
\end{proof}

It follows from  Lemma~\ref{cut}(i,iii) that if $\mmsa$ is a cut, then all reactions involving $S_i \in \mmsa$ are edges of $G_{\mmsa}$ to/from $S_i$ and they appear exactly once. From (i,v) we find that non-interacting connected graphs have no proper full subgraphs. In particular, if $\mmsa$ is a cut such that  $G_{\mmsa}$  is connected, then $G_{\mmsa}$ has no proper full subgraphs.

\begin{figure}[t]
\centering
\includegraphics[]{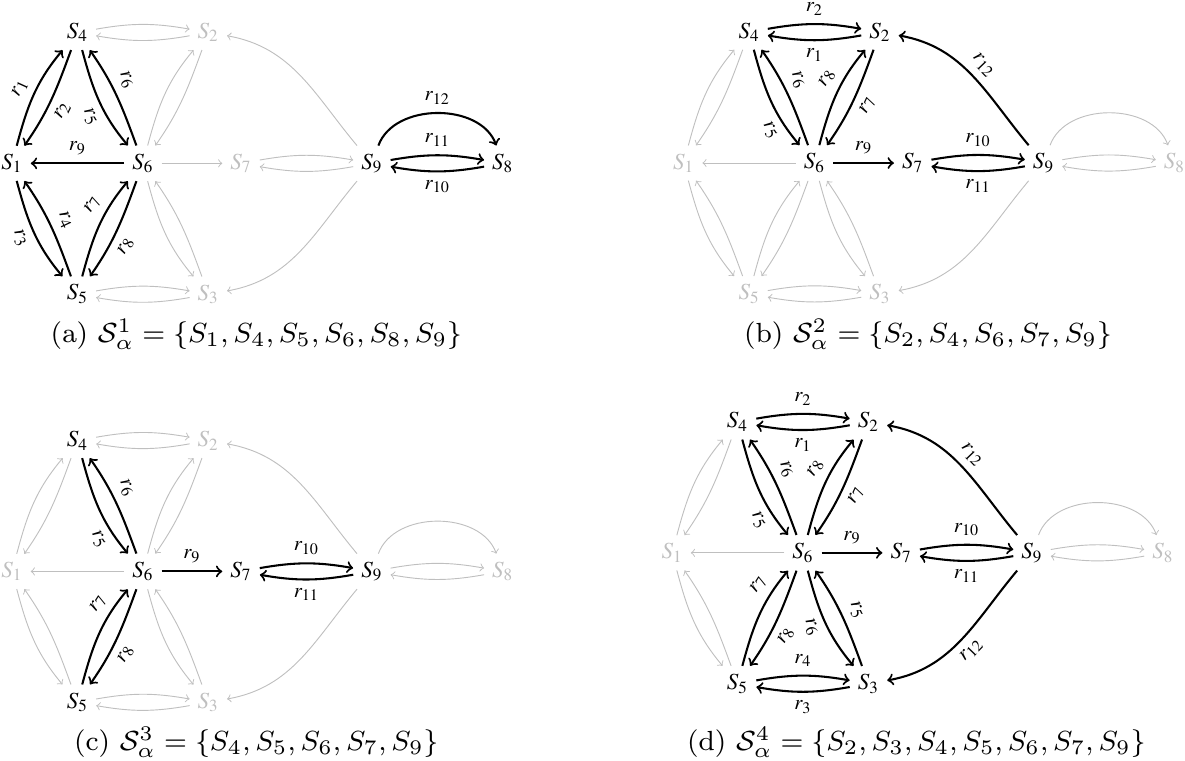}    
\caption{Species subgraphs of the main example. The node sets in (a) and (b) are cuts but the graph in (a) has two connected components and in (b) one; (c) the node set is  non-interacting but not full and cannot be extended to a cut; (d) is a full graph but two of its nodes, $S_2,S_3$ interact. } \label{main:subgraphs}
\end{figure}

\section{Semiflows and the species graph}
In this section we explore the relationship between (P-)semiflows  to full subgraphs of $G_{\mms}$. We come to the main results on semiflows in relation to variable elimination: (1) any semiflow whose support is a cut and the associated graph is connected has all non-zero coefficients equal; (2) the support of a semiflow is non-interacting if and only if it is a cut. 

\begin{lemma}\label{consfull}\label{lemma1.5} Let $\omega$ be a semiflow. 
\begin{enumerate}[(i)]
\item If $\omega$ is minimal, then $G_{\mmS(\omega)}$ is connected.
\item If $\omega$ is a P-semiflow or $\mms(\omega)$ is non-interacting, then $G_{\mms(\omega)}$ is full. 
\item  If  $\omega$ is a P-semiflow or $\mms(\omega)$ is non-interacting, and $G_{\mms(\omega)}$ has no proper full subgraphs, then $\omega$ is minimal. 
\end{enumerate}
\end{lemma}

\begin{proof} 
(i) The idea is that if $G_{\mmS(\omega)}$ is not connected, then the semiflow splits into two, contradicting minimality. If $G_{\mmS(\omega)}$ is not connected, then $G_{\mmS(\omega)}=G_1\cup G_2$ with $G_1,G_2$ being two non-empty disjoint subgraphs of $G_{\mmS(\omega)}$. Let $\mmS_1$, $\mmS_2$ be the node sets of $G_1,G_2$, respectively.  If $\omega=\sum_{i=1}^s \lambda_i S_i$, let $\omega_k=\sum_{i| S_i\in \mms_k} \lambda_i S_i\neq 0$, $k=1,2$.  Since $\mmS_1\cap \mmS_2=\emptyset$, $\omega=\omega_1+\omega_2$. By hypothesis,  $\omega\cdot (y'-y)=0$ for all $y\rightarrow y'\in \mmR$ and thus
$0 = \omega_1\cdot (y'-y)+\omega_2\cdot (y'-y)$. Since $G_1,G_2$ are disjoint, there is no edge between a node in $\mmS_1$ and a node in $\mmS_2$. If $y\rightarrow y'$ is a reaction with $y'_i-y_i\neq 0$ for some $S_i\in \mms_1$, then   $y_{j}=y'_{j}=0$ for all $S_j\in \mmS_2$ and trivially $\omega_2\cdot (y'-y)=0$. Thus  $\omega_1\cdot (y'-y)=0$.  By symmetry, we have that  $\omega_1\cdot (y'-y)=\omega_2\cdot (y'-y)=0$ for all reactions $y\rightarrow y'$ and hence $\omega_k$ is a semiflow for $k=1,2$.

(ii) Let $S_i\in \mmS(\omega)$ and $r\colon y\rightarrow y'$ be a reaction with either $y_{i}$ or $y'_{i}\neq 0$.  We assume that $y_{i}\neq 0$ and the other case follows by symmetry. We want to show that   $r$ is a label in $G_{\mmS(\omega)}$, that is,  $y'_{j}\neq 0$ for some $S_j\in \mmS(\omega)$. By hypothesis, $\omega\cdot (y'-y)=0$. If $\omega=(\lambda_1,\ldots,\lambda_s)$ is a P-semiflow (that is, $\lambda_i\geq 0$), we have $\omega\cdot y' = \omega\cdot y\geq \lambda_iy_i>0$. If $\mms(\omega)$ is non-interacting, then $y_j=0$ for any $i\neq j$ such that $S_j\in \mms(\omega)$ and hence $\omega\cdot y' = \omega\cdot y=\lambda_iy_i\neq 0$.
Either case $\lambda_jy_j'\neq 0$ for some $j$ and hence $y_j'\neq 0$ for some $S_j\in \mmS(\omega)$. Note that the case $j=i$ is accepted. 

 (iii) Assume that $\omega$ is not minimal.  Then by Lemma \ref{terminal}(i) there exists a semiflow $\widetilde{\omega}$ such that $\mms(\widetilde\omega)\subsetneq \mms(\omega)$.  If $\omega$ is a P-semiflow, then by Lemma \ref{terminal}(ii) we can assume that $\widetilde\omega$ is a P-semiflow. If $\mms(\omega)$ is non-interacting, then so is $\mms(\widetilde\omega)$.  Using (ii), $G_{\mms(\widetilde\omega)}$ is full and thus a proper full subgraph of $G_{\mms(\omega)}$,  reaching a contradiction.
\end{proof}
 
Consider the P-semiflows $\omega^i$ \eqref{main:conslaws} of the main example and the subsets $\mmsa^i$ in Figure \ref{main:subgraphs}. Here, $\mmS(\omega^1)\cup \mmS(\omega^2)=\mmsa^1$, $\mmS(\omega^3)=\mmsa^2$ and $\mms(\omega^4)=\mmsa^4$, giving full subgraphs. The two components of $G_{\mmsa^1}$ and the graph $G_{\mmsa^2}$ have no proper full subgraphs, while $G_{\mmsa^2}$ is a proper full subgraph of $G_{\mmsa^4}$. Thus, it follows from Lemma \ref{consfull}(iii) that $\omega^1,\omega^2,\omega^3$ are minimal.

Lemma \ref{consfull}(iii) cannot be reversed. Consider for example  the CRN $R_2$ in \eqref{system2}. The vector $\omega=A+B+C$ is a minimal P-semiflow. The associated graph $G_{\mms(\omega)}$ is
 $ \xymatrix{ A \ar@<0.3ex>[r]^{r_1} &  C \ar@<0.3ex>[l]^{r_2} & B \ar[l]_{r_1}}$ and has a proper full subgraph  $ \xymatrix{ A \ar@<0.3ex>[r]^{r_1} &  C \ar@<0.3ex>[l]^{r_2}}$.  Further,  $\omega_2=A+C$ is a P-semiflow for the CRN $R_1$ in \eqref{system2}, but there is not a P-semiflow involving all species. This implies that if $\widetilde{G}$ is a full connected subgraph of $G$ and $G$ corresponds to a P-semiflow, then $\widetilde{G}$ does not necessarily correspond to a P-semiflow. Similarly, if $\widetilde{G}$ corresponds to a P-semiflow, then $G$ does not necessarily  correspond  to one too. 
 
 \begin{proposition}\label{conscut} Let $\mmS_{\alpha}\subseteq\mmS$ be a subset. If $\mmS_{\alpha}$ is a cut,  then
$\omega=\sum_{i| S_i\in \mmS_{\alpha}} S_i$ is a P-semiflow. In  this  case, $\omega$ is minimal if and only if $G_{\mms_{\alpha}}$ is connected.
\end{proposition}
\begin{proof}
The stoichiometric coefficients of  $S_i\in \mmsa$ are either $0$ or $1$ since $\mmS_{\alpha}$ is a cut. Consider  $r\colon y\rightarrow y'\in \mmR$. If $y_i,y'_i=0$ for all $i$ such that  $S_i\in \mmS_{\alpha}$, then clearly $\omega\cdot (y'-y)=0$. Since $\mmS_{\alpha}$ is  non-interacting, if $y_i\neq 0$ for some $S_i\in \mmS_{\alpha}$, then $y_j=0$ for all $S_j\in \mmS_{\alpha}$ such that $i\neq j$.  From Lemma \ref{cut}(ii,iii), $r$ is the label of exactly one edge of $G_{\mmS_{\alpha}}$ connected to $S_i$, and there exists exactly one species $S_j\in \mmS_{\alpha}$ such that $y'_{j}\neq 0$. Thus, $\omega\cdot (y'-y)= y_i - y_j' = 1-1=0$. This shows that $\omega$ is a P-semiflow.

For the second part of the statement, Lemma \ref{consfull}(i) implies that if $\omega$ is minimal then $G_{\mmsa}$ is connected.  On the other hand, if $G_{\mmsa}$ is connected, then by Lemma \ref{cut}(i,v), $G_{\mmsa}$ has no proper full subgraphs and by Lemma \ref{consfull}(iii), $\omega$ is minimal. 
\end{proof}

The reverse is not true: In the reaction system $S_1+S_2\rightarrow  Y_1+Y_2$, $S_3\rightarrow S_1$, $Y_1\rightarrow S_3$, $Y_2\rightarrow S_3$, the vector $\omega=S_1+S_3+Y_1+Y_2$ is a P-semiflow but the set $\{S_1,S_3,Y_1,Y_2\}$ is not a cut. Further, $\omega$ is minimal. Thus, the non-interacting property cannot be read from the coefficients of the P-semiflow. 

\emph{Remark. }  If $\mmsa$ is a cut, we denote the P-semiflow given in Proposition \ref{conscut} by $\omega(\mmsa)=\sum_{i| S_i\in \mmS_{\alpha}} S_i$. The operation $\omega(\cdot)$ defines a map between the set of cuts and the set of P-semiflows such  that  sets with connected associated graphs are mapped to minimal P-semiflows. The operation $\mms(\cdot)$ defines a map between  the set of P-semiflows and the  subsets of $\mms$ with  full associated graphs such that  minimal P-semiflows are mapped to  subsets with connected associated graphs. The map $\mms(\omega(\cdot))$ is the identity.

From Lemma \ref{consfull}(ii) and Proposition \ref{conscut} we derive the following corollary.

\begin{corollary}\label{nonint}
\begin{enumerate}[(i)]
\item Let $\omega$ be a semiflow such that $\mmS(\omega)$ is non-interacting.  Then $\mms(\omega)$ is a cut. 
\item If $\mmsa\subseteq\mms$ is a non-interacting subset and $\omega(\mmsa)$ is not a P-semiflow, then $\mmsa$ is not a cut and there is no semiflow with support $\mmsa$.
\end{enumerate}
\end{corollary}

Therefore, there is a one-to-one correspondence between cuts and P-semiflows with non-interacting support and  non-zero entries equal to one. The results of this section give rise to the following corollary.

\begin{corollary}\label{semiflows}
Let $\mmsa$ be a non-interacting set such that $G_{\mmsa}$ is connected.
\begin{enumerate}[(i)]
\item If $\mmsa$ is a cut, then any semiflow with support included in $\mmsa$ is a multiple of $\omega(\mmsa)$. 
\item If $\mmsa$ is not a cut, then there are no semiflows with support included in $\mmsa$.
\end{enumerate}
\end{corollary}
\begin{proof}
(i) Follows from Proposition \ref{conscut} and Lemma \ref{terminal}. (ii) Follows from Lemma \ref{cut}(i,v), Lemma \ref{consfull}(ii) and the fact that $G_{\mmsa}$ is not full.
\end{proof}

Consider the non-interacting subset $\mmsa^3$ of the main example. The vector $\omega=S_4+S_5+S_6+S_7+S_9$ is not a P-semiflow since $\omega\cdot (S_1+S_2-S_4)=-1\neq 0$. It follows that $\mmsa^3$ is not a cut and there is no semiflow with support included in $\mmsa^3$. Reciprocally, consider the P-semiflow $\omega^1$. Since $\mmsa^1=\mms(\omega^1)$ is non-interacting, it is a cut. Further, since the graph $G_{\mmsa^2}$ is connected and $\mmsa^2$ is a cut, any semiflow involving the species in $\mmsa^2$ is a multiple of $\omega^3$.
 Checking if a set is non-interacting is straightforward from the set of complexes. However, checking that the associated subgraph is full can be tedious. We have shown that the relationship between semiflows and cuts gives easy conditions for determining if a non-interacting set is a cut.

\emph{Remark. } Using Lemma \ref{snit},  we find that  if $\mmsa$ is a cut but $G_{\mmsa}$ is not connected, then $\mmsa$ decomposes into a disjoint union of  cuts, $\mmsa=\mmsa^1\cup \dots \cup \mmsa^r$, such that $G_{\mmsa^i}$ is connected for all $i$. Thus, $\omega(\mmsa)= \omega(\mmsa^1)+\dots+ \omega(\mmsa^r)$ and the P-semiflow $\omega(\mmsa)$ decomposes into a sum of minimal P-semiflows.  Further,  if $G_{\mmsa}$ is not connected, e.g. has two connected components $G_{\mmsa^1}$ and  $G_{\mmsa^2}$, then $\omega(\mmsa^1)-\omega(\mmsa^2)$ is a semiflow with support in $\mmsa=\mmsa^1\cup\mmsa^2$, but it is not a multiple of $\omega(\mmsa)$. It follows that requiring $G_{\mmsa}$ to be connected is a necessary condition in Corollary \ref{semiflows}.

\section{Elimination of variables}\label{sec:elim}
Let $\mmS_{\alpha}\subseteq \mmS$ be a non-interacting subset. Let $G_{\mmS_{\alpha}}$ be the species graph associated to $\mmS_{\alpha}$  and assume that it is connected. By Corollary \ref{semiflows}, either $\mmsa$ is a cut and $\omega(\mmsa)$ is minimal,  or $\mmsa$ is not a cut and there is no semiflow with support included in $\mmsa$.  In what follows we discuss conditions such that the concentrations of the species in $\mmS_{\alpha}$ can be fully eliminated from the steady state equations. The discussion  depends on whether $\mmsa$ is a cut or not.

 For simplicity, we assume that $\mmsa=\{S_1,\dots,S_{m}\}$.
Let $\mmR_{\alpha}^c$ be the set of reactions  not appearing as labels in $G_{\mmS_{\alpha}}$ but involving some $S_i\in \mmsa$, and $\mmR_{\alpha,out}^c(i)$, $\mmR_{\alpha,in}^c(i)$ the sets of reactions in  $\mmR_{\alpha}^c$ involving $S_i\in \mmS_{\alpha}$ in the initial and terminal complexes, respectively. Clearly,  $\mmR_{\alpha}^c=\bigcup_{i=1}^m\mmR_{\alpha,out}^c(i)\cup\mmR_{\alpha,in}^c(i)$. Note that  $\mmR_{\alpha}^c=\emptyset$ if and only if $\mmsa$ is a cut.

We restrict our attention to the steady state equation \eqref{steadystate} for a fixed $S_i\in \mmS_{\alpha}$. 
Using the expression in \eqref{steadystate2} and the fact that the stoichiometric coefficients are one,  we  can write this equation as $\dot{S}_i= X_i-Y_i$ with:
$$X_i= \sum_{\substack{j=1\\j\neq i}}^m \sum_{S_j\xrightarrow{r} S_i} k_{r}  c^{y(r)}  +  \sum_{r\in \mmR_{\alpha,in}^c(i)}  k_{r}  c^{y(r)}, \quad
Y_i =\sum_{\substack{j=1\\j\neq i}}^m\sum_{S_i\xrightarrow{r} S_j} k_{r}  c^{y(r)}  + \sum_{r\in \mmR_{\alpha,out}^c(i)} k_{r}  c^{y(r)},$$
 where the first summand in each term is taken over the edges in $G_{\mmS_{\alpha}}$. Recall that there can be more than one edge between two species. Further, since the stoichiometric coefficient of $S_i$ is $0$ or $1$, any edge $S_i\rightarrow S_i$ provides no summand in equation \eqref{steadystate}.
 
 Let $C_{\alpha}=C(\mmsa)=\{c_i|S_i\in \mmsa\}=\{c_1,\dots,c_m\}$.
Each of the monomials $c^{y(r)}$ in  $Y_i$ involves $c_i$, and if another $c_k$ is involved, then $S_k$ interacts with $S_i$ (and in particular $k\notin\{1,\dots,m\}$). 
Similarly,  $c^{y(r)}$ in  $X_i$ involves the variable $c_k$ if and only if $S_k$ produces $S_i$. Further, for $r\in \mmR_{\alpha,in}^c(i)$,   $c^{y(r)}$  does not involve any $c_k\in C_{\alpha}$, while  for $S_j\xrightarrow{r} S_i$ with $S_j\in \mmS_{\alpha}$,  the only such variable is $c_j$. 
 It  follows that the system is linear in $C_{\alpha}$ with coefficients in $\R[\Con\cup\, C^c(\mmsa)]$, where 
 \begin{equation}\label{calpha}
C^c(\mmsa)=\{c_i|S_i\in \mms\setminus\mmsa \ \textrm{interacts with or  produces  some }S_j\in \mmsa\}.
\end{equation}
We write $C_{\alpha}^c=C^c(\mmsa)$ for short  and note that $C_{\alpha}\cap C_{\alpha}^c=\emptyset$. 
Let $\widehat{y}_j(r)$ denote the vector in $\R^{s-1}$ obtained from $y(r)$ by removing the $j$-th coordinate. We have shown that  $X_i,Y_i$ can be written so that equation \eqref{steadystate} for $S_i\in \mmsa$ becomes
\begin{equation}\label{redsys} 0= \sum_{j=1}^m a_{i,j}c_j + z_i,\end{equation}
where $a_{i,i}=e_i + d_i$ and
\begin{align*}
e_{i} &= -\sum_{\substack{j=1\\j\neq i}}^m\sum_{S_i\xrightarrow{r} S_j} k_{r}  c^{\widehat{y}_i(r)}, & a_{i,j} &= \sum_{S_j\xrightarrow{r} S_i} k_{r} c^{\widehat{y}_j(r)},  \\  d_i &=   -\sum_{r\in \mmR_{\alpha,out}^c(i)}  k_{r}c^{\widehat{y}_i(r)}, &
  z_i &= \sum_{r\in \mmR_{\alpha,in}^c(i)}  k_{r}c^{y(r)}.
\end{align*}
Let $A=\{a_{i,j}\}$ be the $m\times m$ matrix with $a_{i,j}$ defined as above, $d=(d_1,\dots,d_m)$ and $z=(z_1,\dots,z_m)$.
Note that $\mmsa$ is a cut if and only if $z=d=0$.   The discussion above provides a proof of the following lemma:

\begin{lemma}\label{Slinear2}
The steady state equations \eqref{steadystate} for $S_i\in \mmS_{\alpha}$ form an $m\times m$
  linear system of equations in  $C_{\alpha}$, $Ax+z=0$, where the entries of the matrix $A$ and the independent term $z$ are either zero or S-positive in $\R[\Con\cup\,C_{\alpha}^c]$. Further,   $\mmsa$ is a cut if and only if $z=d=0$, in which case the system is homogeneous.
\end{lemma}

If $A$ has maximal rank $m$, then the system has a unique solution in $\R(\Con\cup\, C_{\alpha}^c)$.  
By Corollary \ref{semiflows},  if $\mmsa$ is a cut then the column sums of $A$ are all zero and the system cannot have maximal rank. If $\mmsa$ is not a cut then there are no semiflows with support in $\mmsa$. 
The column sums of the matrix $A$ are  (for column $i$):
\begin{align*}
\sum_{k=1}^m  a_{k,i}  &= \sum_{i\neq k}\sum_{S_i\xrightarrow{r} S_k} k_{r}  c^{\widehat{y}_i(r)} -\sum_{\substack{j=1\\j\neq i}}^m\sum_{S_i\xrightarrow{r} S_j} k_{r}  c^{\widehat{y}_i(r)} +d_i = d_i.
\end{align*}
These are zero as polynomials in $\R[\Con\cup\, C_{\alpha}^c]$ if and only if $\mmR_{\alpha,out}^c(i)=\emptyset$ for all $i$, and the condition $d=0$ is equivalent to the column sums being zero.
If $\mmsa$ is not a cut, but $d=0$, then $z\neq 0$ as a tuple with entries in $\R[\Con\cup\, C_{\alpha}^c]$. It follows that the system is incompatible  in $\R(\Con\cup\, C_{\alpha}^c)$, because $\sum_i z_i\neq 0$.  The only possible non-negative steady state solutions must satisfy $z_i=0$ for some $i$ and hence $c_j=0$ for some $c_j\in C_{\alpha}^c$ such that $S_j$ produces $S_i\in \mmsa$. 

We proceed now to discuss the case in which $\mmsa$ is a cut and the case in which it is not a cut. Both cases could be merged into a single approach, but the discussion of the first  situation becomes more transparent when it is treated separately.

\medskip
\textbf{Elimination of variables in a cut. } Let $\mmsa\subseteq \mms$ be a cut such that $G_{\mmsa}$ is connected. 
For $S_i\in \mmS_{\alpha}$ the equations in \eqref{steadystate} form an $m\times m$ homogeneous linear system of equations with variables $C_{\alpha}$ and coefficients in $\R[\Con\cup\, C_{\alpha}^c]$. Using equation \eqref{redsys}, equation  \eqref{steadystate} becomes
\begin{equation}\label{B-system}
0= \sum_{j=1}^{m} a_{i,j}c_j, \qquad i=1,\dots,m.
\end{equation}

Because the column sums of $A$ are zero, $A$ is the Laplacian of a labeled directed graph $\wG_{\mmS_{\alpha}}$ with node set $\mmS_{\alpha}$ and a labeled edge $S_j\xrightarrow{a_{i,j}} S_i$, whenever   $a_{i,j}\neq 0$, $i\not=j$. Note that $a_{i,j}\in \R[\Con\cup\, C_{\alpha}^c]$ is S-positive.
We have that $G_{\mmS_{\alpha}}$ is (strongly) connected if and only if $\wG_{\mmS_{\alpha}}$ is. 
The two graphs differ in the labels and in that multiple directed edges from $S_i$ to $S_j$ in  $G_{\mmS_{\alpha}}$  are collapsed to a single directed edge in $\wG_{\mmS_{\alpha}}$. Further, the graph $\wG_{\mmS_{\alpha}}$ has no self-loops (that is, those of $G_{\mmS_{\alpha}}$ are removed by construction).

By the Matrix-Tree theorem, the principal minors $A_{(i,j)}$ of $A=\mathcal{L}(\wG_{\mmS_{\alpha}})$ are 
$$A_{(i,j)} =(-1)^{m-1+i+j}   \sum_{\tau \in \Theta(S_j)}  \pi(\tau).   $$ 
Thus, $A$ has rank $m-1$ if and only if there exists at least one spanning tree in $\wG_{\mmS_{\alpha}}$ rooted at some $S_j$,  $j=1,\dots,m$. 
The next proposition follows from the discussion above. In particular, it holds  if $\wG_{\mmS_{\alpha}}$ (or equivalently $G_{\mmS_{\alpha}}$) is strongly connected.

\begin{proposition}\label{maxrank}
Assume that $\mmS_{\alpha}$ is a cut such that $G_{\mmsa}$ is connected and let $\overline{\omega}=\sum_{i=1}^m c_i$ be the conservation law obtained from  the P-semiflow $\omega(\mmsa)$. The following statements are equivalent:
\begin{enumerate}[(i)]
\item  $\overline{\omega}=\sum_{i=1}^m c_i$ is the only  conservation law with variables in $C_{\alpha}$.
\item  $\wG_{\mmS_{\alpha}}$   has at least one rooted spanning tree. 
\item The rank of $A$ is $m-1$.
\end{enumerate}
\end{proposition}

Since $G_{\mmS_{\alpha}}$ is connected and $\mmsa$ is a cut, any semiflow with support in $\mmsa$ is a multiple of $\omega(\mmsa)$.  The  proposition says that $\wG_{\mmS_{\alpha}}$ has a rooted spanning tree if and only if there are no other conservation laws  with concentrations only in $C_{\alpha}$.

\emph{Remark.}
Let $\mmC_{\alpha}$ be the set of complexes involving at least one species in $\mmS_{\alpha}$. Consider the linkage classes in  $\mmC_{\alpha}$ given by the relation ``ultimately reacts to'' (De\-fi\-nition \ref{maindefs}).  If there is a  rooted spanning tree, then  the root must be in a terminal strong linkage class, because the elements in such a class cannot react to complexes outside  the class.  Further, using the same reasoning,  there cannot be two terminal strong linkage classes.  Consequently, if there is a rooted spanning tree, there is  only  one terminal strong linkage class.
This remark is closely related to  Lemma \ref{stoichkin}.

For simplicity we assume that there exists a spanning tree rooted at $S_1$.  Then, the variables $c_2,\dots,c_{m}$ can be solved in the coefficient field $\R(\Con\cup\, C_{\alpha}^c\cup \{c_1\})$. 
In particular,  using Cramer's rule  and the Matrix-Tree theorem, we obtain
\begin{equation}\label{Srational}
 c_j = \frac{(-1)^{j+1}A_{(1,j)}}{A_{(1,1)}} = \frac{\sigma_j(C_{\alpha}^c)}{\sigma_1(C_{\alpha}^c)} c_1=
\varphi_j(C_{\alpha}^c)c_1, \textrm{ where }
\sigma_j(C_{\alpha}^c)=  \sum_{\tau \in \Theta(S_j)}  \pi(\tau)
\end{equation} 
and $j=1,\dots, m$.  Since there is a spanning tree rooted at $S_1$, it follows that  $\sigma_1(C_{\alpha}^c)$ is S-positive and $\sigma_j(C_{\alpha}^c)$ is either zero or  S-positive in $\R[\Con\cup\, C_{\alpha}^c]$. 
If the graph $\wG_{\mmS_{\alpha}}$  is strongly connected, then $\sigma_j(C_{\alpha}^c) \neq 0$ for all $j$ and  any choice of root $S_j$ could be used instead of $S_1$. 
The arguments given above and the definition of $a_{i,j}$ provide a proof of the following lemma.

\begin{lemma}\label{appear}
If $c_k\in C_{\alpha}^c$ is a  variable of the function $\sigma_j(C_{\alpha}^c)$ for some $j$, 
then there exists $S_i\in \mmS_{\alpha}$ that interacts with $S_k$ and $S_i$ ultimately produces $S_j$ via $\mmsa$. 
Specifically, there is a complex $y^{1}$ involving $S_i$ and $S_k$, and a complex $y^{2}$ involving some species $S_u\in \mmsa$, such that $y^1$ reacts to $y^2$ and $S_u$ ultimately produces $S_j$ via $\mmsa$. If $\wG_{\mmS_{\alpha}}$ is strongly connected, then the reverse is true.
\end{lemma}

The sum of the concentrations in  $C_{\alpha}$ is conserved. Using the equation  $\overline{\omega}=\sum_{i=1}^m c_i$, we obtain
$$ \overline{\omega} = (1+ \varphi_2(C_{\alpha}^c)+\dots+ \varphi_{m}(C_{\alpha}^c))c_1,$$ 
where the coefficient of $c_1$ is S-positive in $\R(\Con\cup\, C_{\alpha}^c)$. Thus,
$$ c_1 = \overline{\varphi}_1(C_{\alpha}^c)=\frac{\overline\omega}{1+ \varphi_2(C_{\alpha}^c)+\dots+ \varphi_{m}(C_{\alpha}^c)}=\frac{\overline\omega\sigma_1(C_{\alpha}^c)}{\sigma_1(C_{\alpha}^c)+ \sigma_2(C_{\alpha}^c)+\dots+ \sigma_m(C_{\alpha}^c)},$$ 
with $\overline{\varphi}_1$ being an S-positive rational function in $C_{\alpha}^c$ with coefficients in $\R(\Con\cup \{\overline{\omega}\})$. Observe that $\overline\omega$ becomes an extra parameter  and can be treated as a symbol as well. Further, if  $\overline{\omega}$ is assigned a positive value, then  
$c_1> 0$ at steady state for positive values of $C_{\alpha}^c$.
 By substitution of $c_1$ by $\overline{\varphi}_1$, we  obtain 
\begin{equation}\label{substcons} 
c_j= \overline{\varphi}_j(C_{\alpha}^c):= \varphi_j(C_{\alpha}^c)\overline{\varphi}_1(C_{\alpha}^c), \qquad j=2,\dots,m,
\end{equation}
 with $\overline{\varphi}_j$  being either zero or an S-positive rational function in $C_{\alpha}^c$ with coefficients in  $\R(\Con\cup \{\overline{\omega}\})$.

\begin{proposition}\label{resultS1} Let $\mmsa\subseteq\mms$ be a cut such that $G_{\mmsa}$ is connected.
Assume that   there is a spanning tree of $\wG_{\mmS_{\alpha}}$  rooted at some species $S_{i}$. Then, there exists a zero or S-positive rational function $\varphi_{j}$ in $C_{\alpha}^c$ with coefficients in $\R(\Con)$, such that equation \eqref{steadystate} for $c_j\in C_{\alpha}$ is satisfied  in $\R(\Con\cup\, C_{\alpha}^c)$ if and only if
$$ c_{j}=\varphi_{j}(C_{\alpha}^c) c_{i},\qquad c_j\in C_{\alpha}.$$
Further, there exists an S-positive rational function $\overline{\varphi}_{i}$ in $C_{\alpha}^c$ with coefficients in $\R(\Con \cup \{\overline{\omega}\})$, such that  the conservation law $\overline{\omega}=\sum_{k=1}^m  c_k$  is fulfilled if and only if $c_{i}= \overline{\varphi}_{i}(C_{\alpha}^c).$
\end{proposition}

Consider the main example and the cut $\mmsa=\{S_1,S_4,S_5,S_6\}$ corresponding to a connected component of the graph $G_{\mmsa^1}$ in Figure \ref{main:subgraphs}(a). System \eqref{B-system}  
becomes
$$\left(\begin{array}{cccc}
 -k_1c_2- k_3c_3 & k_2  & k_4 & k_9 \\
 k_1 c_2 &  - k_2 - k_5c_3 & 0 &  k_6 \\
 k_3c_3 & 0 & - k_4-k_7c_2 & k_8 \\
0 &  k_5c_3 &  k_7c_2  & - k_6  - k_8-k_9
\end{array}\right)\left(\begin{array}{c} c_1 \\ c_4 \\ c_5 \\ c_6\end{array}\right) =0.$$
The column sums are zero because of the conservation law $\overline{\omega}^1=c_1+c_4+c_5+c_6$. 
The graph $\widehat{G}_{\mmsa}$ is strongly connected (Figure \ref{main:elimgraphs}(a)),  as  is observed in many real (bio)chemical systems. Thus rooted spanning trees exist and the system has rank $3$. 
 This also follows from Proposition \ref{maxrank} and Lemma \ref{stoichkin}, since each linkage class of the CRN has exactly one terminal strong linkage class. 
 
 The polynomials $\sigma_j$ are:
\begin{align*}
\sigma_1 &= k_2k_4(k_6+k_8+k_9) + k_2k_7(k_6+k_9)c_2 + k_4k_5(k_8+k_9)c_3  + k_5k_7 k_9c_2c_3 \\
 \sigma_4 &= k_1 k_4(k_6+k_8+k_9)c_2 +  k_1k_7(k_6+k_9) c_2^2+ k_3 k_6 k_7c_2 c_3\\ 
 \sigma_5 &= k_2 k_3(k_6+k_8+k_9)c_3 +  k_3 k_5(k_8+k_9)c_3^2+ k_1 k_5 k_8 c_2 c_3 \\
 \sigma_6 &=(k_1 k_4 k_5 + k_2 k_3 k_7)c_2 c_3  +  k_1 k_5 k_7 c_2^2 c_3 +  k_3 k_5 k_7c_2 c_3^2.
\end{align*}
Each monomial in $\sigma_j$ corresponds to a spanning tree rooted at $S_j$. The species $S_2,S_3$ are the only  species interacting with a species in $\mmsa$ and thus only $c_2,c_3$ appear in the expressions. 
Using \eqref{Srational} and \eqref{substcons} we find the steady state expressions of $c_1,c_4,c_5,c_6$ in terms of the rate constants, the total amount $\overline{\omega}^1$, and  the concentrations $c_2,c_3$.

\emph{Remark. } It is  straightforward to find $\sigma_j$ by computing the principal minors of $A$ using any computer algebra software. The advantage of the Matrix-Tree description in the theoretical discussion is that S-positivity of the solutions is easily obtained.

\medskip
\textbf{Elimination of variables in a subset that is not a cut.}
Let $\mmS_{\alpha}\subseteq \mmS$ be a non-interacting subset that is not a cut and assume that $G_{\mmS_{\alpha}}$  is connected. As discussed above, if the column sums are  zero then there are no positive steady state solutions. 
If the column sums are not all zero, then the matrix  $A$ is not a Laplacian. However,  $A$ can be extended such that its determinant is a principal minor of a Laplacian.

Consider  the labeled directed graph $\widehat{G}_{\mmS_{\alpha}}$ with node set  $\mmS_{\alpha} \cup \{*\}$.  We order the nodes such that  $S_i$ is  the $i$-th node and
 $*$ the $(m+1)$-th  node.  The graph $\widehat{G}_{\mmS_{\alpha}}$ has the following labeled directed edges:
$S_j\xrightarrow{a_{i,j}} S_i$ if $a_{i,j}\neq 0$ and $i\neq j$,  $S_i\xrightarrow{-d_i}  *$ if $d_i\neq 0$, and 
  $*\xrightarrow{z_i} S_i$ if $z_i\neq 0$. All labels are S-positive in $\R[\Con\cup\, C_{\alpha}^c]$.
Let $\mathcal{L}=\{\lambda_{i,j}\}$ be the Laplacian of $\widehat{G}_{\mmS_{\alpha}}$.  If $i,j\leq m$, then $\lambda_{i,j}=a_{i,j}$. The entries of the last row  are $\lambda_{m+1,i} =- d_i$ for $i\leq m$ and the entries of the last  column are $\lambda_{i,m+1}= z_i$ for $i\leq m$. 
We conclude that the $(m+1,m+1)$ principal minor of $\mathcal{L}$ is exactly $A$ and thus,  by the Matrix-Tree theorem, we have
$$\sigma(C_{\alpha}^c):=  (-1)^m \det(A) = (-1)^m \mathcal{L}_{(m+1,m+1)} =   \sum_{\tau \in \Theta(*)}  \pi(\tau).  $$ 
If there exists at least one spanning tree rooted at $*$, then $(-1)^{m}\det(A)$ is S-positive in $\R[\Con\cup C_{\alpha}^c]$. In this case the system $Ax+z=0$ has a unique solution in $\R(\Con \cup\, C_{\alpha}^c)$. A spanning tree rooted at $*$ exists if and only if for all  species  $S_i\in \mmS_{\alpha}$, there exists a reaction $y\rightarrow y'$ such that $y'$ does not involve any species in $\mmsa$, $y$ involves some $S_u\in \mmsa$ and $S_i$ ultimately produces  $S_u$.  The existence of such a spanning tree ensures that $\mmR_{\alpha,out}^c(i)\neq \emptyset$   and thus $d_i\neq 0$ for some $i$.

\begin{figure}[t]
\centering
\includegraphics[]{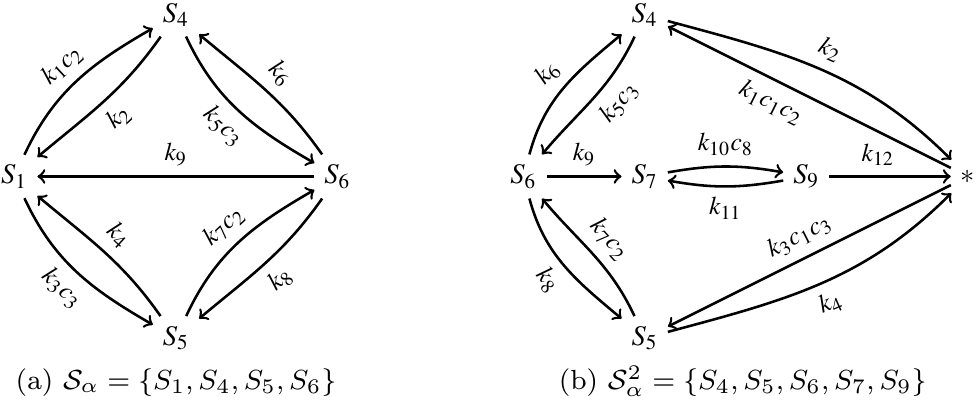}   
 \caption{(a) The graph $\widehat{G}_{\mmsa}$ in the main example for the cut $\mmsa$ and (b) The graph $\widehat{G}_{\mmsa}^2$ for the non-interacting set $\mmsa^2$,  which is not a cut. } \label{main:elimgraphs}
\end{figure}

 Since $A$ is non-interacting but not a cut, there are no semiflows with support in $\mmsa$, Corollary \ref{semiflows}(ii). Similarly to Proposition \ref{maxrank}, we obtain the following proposition:
\begin{proposition}\label{maxrank2}
Assume that $\mmS_{\alpha}$ is a non-interacting set that is not cut and such that $G_{\mmsa}$ is connected. Then  $A$ has maximal rank if and only if there exists a spanning tree  rooted at $*$. Further, if $A$ has maximal rank then there are no conservation laws in the concentrations in $C_{\alpha}$.
\end{proposition}

 If  $A$  does not have maximal rank, then there is a vanishing linear combination of the rows of $A$, $0=\sum_{k=1}^m\lambda_k a_{k,j}$ for all $j$. If there are no conservation laws  in the concentrations in $C_{\alpha}$,
then $0\neq \sum_{k=1}^m\lambda_k \dot{c_k}= \sum_{k=1}^m \sum_{j=1}^m \lambda_k (a_{k,j}c_j + z_k)$ and  it follows that  
$\sum_{k=1}^m\lambda_k z_{k}\neq 0$. If $\lambda_k\geq 0$ for all $k$, then we conclude  that the system \eqref{redsys} is  incompatible in $\R(\Con\cup\, C_{\alpha}^c)$ and there are no positive steady states.  

Assume that a spanning tree rooted at $*$ exists.
For $i=1,\dots,m$, let $\sigma_i$ be the following polynomial in $C_{\alpha}^c$,
$$\sigma_i(C_{\alpha}^c)=(-1)^{i+1} \mathcal{L}_{(m+1,i)} =  \sum_{\tau \in \Theta(S_i)}  \pi(\tau),$$ 
which is either zero or  S-positive in $\R[\Con \cup\, C_{\alpha}^c]$. 
 By Cramer's rule, we have
$$ c_i = \varphi_i(C_{\alpha}^c)=\frac{(-1)^{m+1-i} \mathcal{L}_{(m+1,i)}  }{(-1)^m \mathcal{L}_{(m+1,m+1)}} = \frac{\sigma_i(C_{\alpha}^c)}{\sigma(C_{\alpha}^c)}, $$ 
which is either zero or  S-positive in $\R(\Con \cup\, C_{\alpha}^c)$. 
 If there exists at least one spanning tree rooted at $S_i$, then $\sigma_i\neq 0$ as a polynomial in $\R[\Con \cup\, C_{\alpha}^c]$. A necessary condition for $\sigma_i\neq 0$ is the existence of a directed path from $*$ to $S_i$, which implies that $S_i$ is ultimately produced from some species $S_k\in \mmS\setminus \mmS_{\alpha}$.
  In particular, if $\widehat{G}_{\mmS_{\alpha}}$ is strongly connected then all concentrations are non-zero as elements in $\R(\Con \cup\, C_{\alpha}^c)$.  

Consider the set $\mmsa=\mmsa^3=\{S_4,S_5,S_6,S_7,S_9\}$ in the main example. It is non-interacting, not a cut, and $G_{\mmsa}$ is connected. Further, all species ultimately produce $S_9\in \mmsa$ and $S_9$ reacts to  $S_2+S_3+S_8$ which does not involve species in $\mmsa$.  Hence a spanning tree rooted at $*$ exists. The graph $\widehat{G}_{\mmsa}$ is depicted in Figure \ref{main:elimgraphs}(b) and is strongly connected. We have that $z=(k_1c_1c_2, k_3c_1c_3, 0, 0, 0)$,$C_{\alpha}^c=\{c_1,c_2,c_3,c_8\}$ and
\begin{align*}
\sigma =&  k_{10}k_{12} (k_2 k_4 (k_6+k_8+k_9) +  k_2 k_6 k_7c_2  +  \\ & k_4 k_5 k_8c_3 + 
   k_9(k_4 k_5c_3 + k_2 k_7c_2 +  k_5 k_7c_2 c_3))c_8 \\
   \sigma_4=&  k_{10} k_{12}(k_1 k_4 (k_6+k_8+k_9) + k_1  k_7(k_6+k_9) c_2 +  k_3 k_6 k_7c_3) c_1 c_2 c_8\\
\sigma_5 =&  k_{10} k_{12} (k_2 k_3 (k_6 + k_8+k_9) + k_1 k_5 k_8  c_2+ k_3 k_5 (k_8+k_9) c_3)c_1 c_3 c_8 \\
 \sigma_6 =&  k_{10} k_{12} (k_1 k_4 k_5 + k_2 k_3 k_7 +  k_1 k_5 k_7 c_2+ k_3 k_5 k_7c_3)c_1 c_2 c_3 c_8 \\
 \sigma_7 =& k_9(k_{11} + k_{12}) (k_1 k_4 k_5 + k_2 k_3 k_7 + k_1 k_5 k_7 c_2 + 
   k_3 k_5 k_7 c_3) c_1 c_2 c_3  \\
  \sigma_9 =&   k_9k_{10} (k_1 k_4 k_5 + k_2 k_3 k_7 + k_1 k_5 k_7 c_2 + k_3 k_5 k_7 c_3) c_1 c_2 c_3 c_8. 
\end{align*}
The concentration $c_1$ is only in the label of out-edges from $*$ and thus $c_1$ is not  in $\sigma$.

\begin{proposition}\label{resultS2}
Assume that   there is a spanning tree of $\widehat{G}_{\mmS_{\alpha}}$  rooted at  $*$. Then, there exists a zero or S-positive rational function $\varphi_{i}$ in $C_{\alpha}^c$ with coefficients in $\R(\Con)$, such that equation \eqref{steadystate} for $c_i\in C_{\alpha}$ is satisfied  in $\R(\Con\cup\, C_{\alpha}^c)$ if and only if $c_{i}=\varphi_{i}(C_{\alpha}^c)$.
\end{proposition}

\emph{Remark. } The procedure outlined here can be stated in full generality: Consider a square linear system of equations $Ax+z=0$, such that the entries of $z$ and the off-diagonal entries of $A$ are positive and  the column sums of $A$ are zero or negative. If $A$ has maximal rank, then the unique solution of the system is non-negative.

\emph{Remark. } If the matrix $A$ does not have maximal rank, then we can always selecxt a subset of $\mmsa$ such that the corresponding matrix has maximal rank and proceed with elimination of the variables in the subset.

\begin{table}[!b]
\begin{tabular}{p{2.7cm}|p{5cm}|p{5cm}}
 &  $\mmsa$ cut  & $\mmsa$ not a cut   \\ \hline
Characterization & $\omega(\mmsa)$  semiflow or  $z=d=0$ &  $\nexists$ semiflow  or $z-d\neq 0$ \\   
\vspace{0.01cm} Elimination of $C_{\alpha}$ works if in $\wG_{\mmsa}$ ...   
& \vspace{0.01cm}  $\exists$  rooted spanning tree  (equivalent to $\overline{\omega}=\sum_{c_i\in C_{\alpha}} c_i$  being the ``only'' conservation law in $C_{\alpha}$) & 
 \vspace{0.01cm}  $\exists$ spanning tree rooted at  $*$ (implies $\nexists$ conservation law in $C_{\alpha}$ and $d\neq 0$) 
\end{tabular}
\caption{Summary of the conditions required for the variable elimination procedure for a non-interacting set $\mmsa$. Here  ``only'' indicates up to multiplication 
by a constant.
}\label{table1}
\end{table} 

\medskip
{\bf Remarks.}
Let $\mmsa\subseteq \mms$ be any non-interacting subset such that $G_{\mmsa}$ is connected. 
We have proven that for all $c_i\in C_{\alpha}$,  there exists a rational function $\varphi_i$ such that $c_i=\varphi_i(C_{\alpha}^c)$ at steady state, provided  some spanning trees exist. 
When $\mmsa$ is a cut,  the P-semiflow $\omega(\mmsa)$ is required in the elimination, while when $\mmsa$ is not a cut, variables are eliminated using only the steady state equations.
 
If $G_{\mmsa}$ is not connected, then the results above apply to the connected components separately, since the underlying node set of each connected component is non-interacting. Further, let $\mmsa^1,\mmsa^2$ be two non-interacting sets such that $G_{\mmsa^1}$ and $G_{\mmsa^2}$ are disjoint and connected. One easily sees that $C(\mmsa^2)\cap C^c(\mmsa^1)=\emptyset$, that is, both sets of variables $C(\mmsa^1)$ and $C(\mmsa^2)$ can be simultaneously eliminated. Additionally, if we let $\mmsa=\mmsa^1\cup\mmsa^2$  then  $C^c(\mmsa)=C^c(\mmsa^1)\cup C^c(\mmsa^2)$.
For instance, consider  $\mmsa=\{S_1,S_4,S_5,S_6,S_8,S_9\}$ in Figure \ref{main:subgraphs}(a). The associated graph has two connected components, which are strongly connected. The concentrations $c_i\in C_{\alpha}$ can be expressed as S-positive rational functions in $C_{\alpha}^c=\{c_2,c_3,c_7\}$.

The conditions to apply the variable elimination procedure are summarized in Table \ref{table1}.
The procedure guarantees that  if positive values are assigned to  all $c_j\in C_{\alpha}^c$, then $c_i$ is non-negative. For $c_i$ to be positive, that is, $\sigma_i\neq 0$, the existence of at least one in-edge to $S_i$ is  necessary. Otherwise the concentration at steady state of $S_i$ is zero, which is  expected if $S_i$ is only consumed and never produced.  Further:

\begin{proposition}\label{snonzero2}
Let $\mmsa$ be a non-interacting subset of $\mms$ such that the concentrations $C_{\alpha}$ can be eliminated from the steady state equations.  Each component of the graph $\wG_{\mmS_{\alpha}}$  is strongly connected if and only if any steady state solution  satisfies $c_j>0$ for all $c_j\in C_{\alpha}$  (and for any positive total amount if appropriate), whenever the variables in   $C_{\alpha}^c$ take positive values. 
\end{proposition}

\section{Steady state equations}
Let $\mmsa$ be any non-interacting subset  such that $G_{\mmsa}$ is connected and that the variables in $C_{\alpha}$ can be eliminated by the procedure above. 
Let $\Phi_u(C_{\alpha}^c)=0$ be the equation obtained from $\dot{c}_u=0$, $u=m+1,\dots,s$, after elimination of variables in $C_{\alpha}$ and removal of denominators. The denominators can be chosen to be S-positive and multiplication  by  the  denominators does not  change the positivity of solutions.
Fix a maximal set of $n$ independent combinations $\xi^l(c_1,\dots,c_s)=\sum_{i=1}^s \lambda_i^l c_i$ providing conservation laws that includes those corresponding to the full connected components of $G_{\mmsa}$ (that is, to cuts). For given total amounts, the steady state equations are complemented with the equations $\overline\omega^l= \xi^l(c_1,\dots,c_s)$, $l=1,\dots,n$. If the conservation law corresponds to a cut, then the elimination procedure ensures that $\xi^l(\varphi_1(C_{\alpha}^c),\dots,\varphi_m(C_{\alpha}^c),c_{m+1},\dots,c_s) =\overline \omega^l$  and the equation becomes redundant.
 
\begin{theorem}\label{steadystates}
Consider a CRN with a non-interacting set $\mmS_{\alpha}$. Assume that $\mmsa=\mmsa^1\cup\dots\cup\mmsa^r$ is a partition of $\mmsa$ into disjoint sets such that  $G_{\mmsa^j}$ is connected   and $\wG_{\mmsa^j}$ admits a spanning tree  for all $j$. If $\mmsa^j$ is a cut, assume that the spanning tree is rooted at some $S_i\in \mmsa$ and otherwise assume that it is rooted at  $*$.  Let total amounts $\overline\omega^{l}$ be given for the $n$ conservation laws. 

The non-negative steady states with positive values in $C_{\alpha}^c$ are in one-to-one correspondence with the positive solutions to 
$$\Phi_u(C_{\alpha}^c)=0,\qquad  \overline\omega^{l}=\xi^l(C_{\alpha}^c):=\xi^l(\varphi_1(C_{\alpha}^c),\dots,\varphi_m(C_{\alpha}^c),c_{m+1},\dots,c_s)$$ 
for $u=m+1,\dots,s$ and $l=1,\dots,n$.
\end{theorem}
\begin{proof}
We have shown that any non-negative steady state solution with positive values for  $c_i\in C_{\alpha}^c$  must satisfy these equations. For the reverse, we apply the following to each connected component of $G_{\mmsa}$.  Consider a positive solution $c=(c_{m+1},\dots,c_s)$  to the equations $\Phi_u(C_{\alpha}^c)=0$ and $\overline\omega^{l}=\xi^l(C_{\alpha}^c)$. For $i=1,\dots,m$, define $c_i$ through Proposition \ref{resultS1} or \ref{resultS2}, depending on whether $S_i$ belongs to a cut $\mmsa^j$ or not. For positive rate constants and positive total amounts, $c_i$ is non-negative (because  the rational functions defining it are S-positive). By construction this procedure automatically ensures that  conservation laws corresponding to  cuts are fulfilled. Using Propositions \ref{resultS1} and \ref{resultS2} the values $c_i$ satisfy \eqref{steadystate}. Since $\Phi_u(C_{\alpha}^c)=0$ is the steady state equation $\dot{c_u}=0$ after substitution of the eliminated variables, this equation is also satisfied and the same reasoning applies to the equation $\overline\omega^{l}=\xi^l(c_1,\dots,c_m)$. Thus, $(c_1,\dots,c_m)$  is a solution to the steady state equations and satisfies the conservation laws corresponding to the total amounts $\overline\omega^{l}$. 
\end{proof}

This theorem together with Proposition \ref{snonzero2} give the following corollary.

\begin{corollary} With the conditions of Theorem \ref{steadystates}, assume further that each graph $\wG_{\mmS_{\alpha}^j}$  is  strongly connected. Then, the positive steady states of the system  are in one-to-one correspondence with the positive solutions to  $\Phi_u(C_{\alpha}^c)=0$ and $\overline\omega^{l}=\xi^l(C_{\alpha}^c)$
for $u=m+1,\dots,s$ and $l=1,\dots,n$.  
Further, if a steady state solution satisfies $c_i=0$ for  $c_i\in C_{\alpha}$,  then there exists some $c_j\in C_{\alpha}^c$ such that  $c_j=0$.
\end{corollary}

In the main example,   the set $\mmsa=\{S_1,S_4,S_5,S_6,S_8,S_9\}$ is the largest non-interacting subset of $\mms$ and thus provides the maximal number of linearly eliminated concentrations. The initial steady state system of equations is reduced to three equations: For instance the one corresponding to $\dot{c_7}=0$, and the two conservation laws 
$\overline{\omega}^3 = c_2+c_4+c_6+c_7+c_9$ and $\overline{\omega} = c_3+c_5+c_6+c_7+c_9$ (which corresponds to $\omega^4-\omega^3$).
Because of the conservation laws, the equations $\dot{c_2}=0$ and $\dot{c_3}=0$ are redundant. 
The elimination from cuts provides $c_k=\sigma_k/\sigma_1$ for $k=4,5,6$, $c_1=\overline{\omega}^1\sigma_1/(\sigma_1+\sigma_4+\sigma_5+\sigma_6)$, $c_9=\sigma_9/\sigma_8$
 and $c_8=\overline{\omega}^2\sigma_8/(\sigma_8+\sigma_9)$ 
 with $\sigma_i$ S-positive polynomials in $c_2,c_3,c_7$ and coefficients in $\R(\Con\cup\{\overline{\omega}^1,\overline{\omega}^2\})$ for all $i$. The steady state equations are thus reduced to:
\begin{align*}
0 &= k_9\sigma_6(\sigma_8+\sigma_9)\sigma_8 - k_{10} \overline{\omega}^2\sigma_8^2\sigma_1 c_7  + k_{11} \sigma_1\sigma_9(\sigma_8+\sigma_9) \\ \overline{\omega}^3 \sigma_1\sigma_8 &= \sigma_1\sigma_8c_2+ \sigma_4\sigma_8+\sigma_6\sigma_8+\sigma_1\sigma_8 c_7+\sigma_1\sigma_9 \\ \overline{\omega} \sigma_1\sigma_8 &=  \sigma_1\sigma_8c_3+\sigma_5\sigma_8+\sigma_6\sigma_8+ \sigma_1\sigma_8c_7+\sigma_1\sigma_9.
\end{align*}

\section*{Acknowledgments}
EF is supported by a postdoctoral grant from the ``Ministerio de Educaci\'on'' of Spain and the project  MTM2009-14163-C02-01 from the ``Ministerio de Ciencia e Innovaci\'on''.  CW is supported by the Lundbeck Foundation, Denmark and the Leverhulme Trust, UK.  Part of this work was done while EF and CW were visiting Imperial College London in fall 2011.

\end{document}